\documentclass[12pt]{amsart}
\usepackage{amsfonts}
\usepackage[mathscr]{euscript}

\newcommand{\lbl}[1]{\label{#1}}

\newtheorem{theo}{Theorem}[section]
\newtheorem{prop}[theo]{Proposition}
\newtheorem{lem}[theo]{Lemma}
\newtheorem{rem}[theo]{Remark}

\newcommand{\be}{\begin{equation}}
\newcommand{\ee}{\end{equation}}
\newcommand\bes{\begin{eqnarray}} \newcommand\ees{\end{eqnarray}}
\newcommand{\bess}{\begin{eqnarray*}}
\newcommand{\eess}{\end{eqnarray*}}
\newcommand\bedd{\bess\left\{\begin{array}{ll}\medskip}
\newcommand\eedd{\end{array}\right.\eess}

\newcommand\vp{\varphi}

\newcommand\dd{\displaystyle}

\newcommand\ty{\infty}

\newcommand\rr{\right}

\newcommand\R{\mathbb{R}}

\begin{document}

\author[Y. Du, M.X. Wang \& M. Zhou]{Yihong Du$^\dag$, Mingxin Wang$^\ddag$ and Maolin Zhou$^\S$}
\thanks {$^\dag$ 
School of Science and Technology, University of New England, Armidale, NSW 2351, Australia.\ E-mail: ydu@turing.une.edu.au.}
\thanks{$^\ddag$
Natural Science Research Center,
Harbin Institute of Technology,
Harbin 150080, China. \ E-mail: mxwang@hit.edu.cn.}
\thanks{$^\S$
School of Science and Technology, University of New England, Armidale, NSW 2351, Australia. \ E-mail: zhouutokyo@gmail.com.
}

\title[Semi-wave and spreading speed for the competition model]{Semi-wave and spreading speed for the diffusive competition model with a free boundary}
\thanks {This work was supported by the Australian Research Council and also by NSFC Grant 11371113. Y. Du thanks Prof. Zhigui Lin for some interesting discussions.}
\date{\today}

\begin{abstract} We determine the asymptotic spreading speed of an invasive species, which invades
the territory of a native competitor,  governed by a diffusive competition model with a free boundary in a spherically symmetric setting. This free boundary problem was studied recently in \cite{DLin1},
but only rough bounds of the spreading speed was obtained there.
We show in this paper that there exists an asymptotic  spreading speed, which is determined by a certain traveling wave type system of one space dimension, called a semi-wave. This appears to be the first result that gives the precise
asymptotic spreading speed for a two species system  with free boundaries.
\end{abstract}

 
\maketitle
%---------------------------------------------------------------------
\def\theequation{\arabic{section}.\arabic{equation}}

\setlength{\baselineskip}{16pt}

 \section{Introduction}
 \setcounter{equation}{0} {\setlength\arraycolsep{2pt}

The main purpose of this paper is to determine the asymptotic spreading speed of an invasive species, which invades
the territory of a native competitor,  according to the following diffusive competition model with a free boundary in $\R^N$ ($N\geq 1$), in a spherically symmetric setting:
\begin{align}
\left\{
\begin{array}{lll}
U_{t}-d_1 \Delta U=U(a_1-b_1 U-c_1 V),\; & t>0, \ 0\leq r<H(t),  \\
V_{t}-d_2 \Delta V=V(a_2-b_2 U-c_2 V),\; & t>0, \ 0\leq r<\infty,  \\
U_r(t,0)=V_r(t,0)=0,\; U(t, r)=0,\quad & t>0,\ H(t)\leq r<\infty,\\
H'(t)=-\hat\mu U_r(t,H(t)),\quad &t>0,\\
H(0)=H_0,  U(0, r)=U_{0}(r),\;  &0\leq r\leq H_0,\\
 V(0, r)=V_0(r), &0\leq r<\infty,\\
\end{array} \right.
\label{f1}
\end{align}
where $r=|x|$, $\Delta U=U_{rr}+\frac{N-1}{r}U_r,\ r=H(t)$ is the moving
boundary to be determined, $H_0$, $\hat\mu$, $d_i$,
 $a_i$, $b_i$ and $c_i\; (i=1,2)$ are given positive constants, and the initial functions
$U_0$  and  $V_0$ satisfy
\begin{align}
\left\{
\begin{array}{ll}
U_0\in C^{2}([0, H_0]), \ U_0'(0)=U_0(H_0)=0\ \ \textrm{and} \ U_0>0\ \textrm{in}\ [0, H_0),\\
V_0\in C^{2}([0, +\infty))\cap L^\infty(0, +\infty), \ V_0'(0)=0\ \ \textrm{and} \ V_0\geq,\not\equiv 0\
\textrm{in}\ [0, +\infty).
\end{array} \right.
\label{Ae}
\end{align}
 In this model, the first
species ($U$), which exists initially in the ball $\{r<H_0\}$,  invades into the new territory through
its range boundary $\{r=H(t)\}$ (the invading front), which evolves according to the free boundary condition $H'(t)=-\hat\mu U_r(t, H(t))$, where $\hat\mu$ is a given positive constant.
  The second species ($V$) is
native, which undergoes diffusion and growth in the entire available habitat (assumed to be $\R^N$ here).
The constants $d_1$ and $d_2$ are the diffusion rates of $U$ and $V$,
respectively, $a_1$ and $a_2$ are the intrinsic growth rates, $b_1$ and $c_2$ are the intraspecific and
   $c_1$ and $b_2$ the interspecific competition rates.

Problem \eqref{f1} has been studied in \cite{DLin1} recently. It is shown in \cite{DLin1} that, if $U$ is an inferior competitor, characterized by
\begin{equation}\label{inf}
\frac{a_1}{a_2}<\min\left\{\frac{b_1}{b_2},\frac{c_1}{c_2}\right\},
\end{equation}
then the invasion of $U$  always fails, in the sense that
\[
\lim_{t\to+\infty}(U(t,\cdot), V(t,\cdot))=\left(0,\frac{a_2}{c_2}\right) \mbox{ in } L^\infty_{loc}([0,+\infty)).
\]
On the other hand, if $U$ is a superior competitor, namely,
\begin{equation}\label{sup}
\frac{a_1}{a_2}>\max\left\{\frac{b_1}{b_2},\frac{c_1}{c_2}\right\},
\end{equation}
then the fate of the invasion of $U$ is determined by a spreading-vanishing dichotomy:

Either
\begin{itemize}
\item[(i)]({\it Spreading of $U$}).\ \ $\lim_{t\to+\infty}H(t)=+\infty$ and $\lim_{t\to
+\infty}\big(U(t, r), V(t,r)\big)=\left(\frac{a_1}{b_1}, \, 0\right)$ uniformly in any compact subset of $[0, \infty)$;
\end{itemize}

or
\begin{itemize}
\item[(ii)]({\it Vanishing of $U$}).\ \ $\lim_{t\to+\infty}H(t) <+\infty$\ and \ $\lim_{t\to +\infty}\|U(t,
\cdot)\|_{C([0,H(t)])}=0$, \\ 
$\lim_{t\to\infty}V(t,r)=\frac{a_2}{c_2}$ uniformly in  any compact subset of $[0,\infty)$.
\end{itemize}

Sharp criteria for spreading and vanishing of $U$ are  given in \cite{DLin1} (see Theorem 4.4 there). In particular, under the condition \eqref{sup},
 if $H_0$ is greater than a certain number determined by an eigenvalue problem, then spreading of $U$ always happens. However, when spreading of $U$ happens, only rough estimates for the spreading speed of the invading front $\{r=H(t)\}$ are obtained in \cite{DLin1}. Indeed, determining the precise spreading speed for population {\it systems} with free boundaries has been a difficult problem in general, and this paper appears to be the first to provide a positive answer to this question.

Without a native competitor in the environment  (namely in the case $V\equiv 0$), \eqref{f1} reduces to a free boundary problem for $U$ considered in \cite{DG}, which extends the earlier work \cite{DLin} from one space dimension to the radially symmetric higher space dimension case. In these relatively simpler situations a spreading-vanishing dichotomy is known, and when spreading happens, the spreading speed has been determined through a semi-wave problem involving a single equation.
More general results in this direction can be found in \cite{DGP, DLiang, DLou, DLZ, DMZ1, DMZ2, PZ}, where \cite{DGP, PZ}
considers time-periodic environment, \cite{DLiang} studies space-periodic environment, \cite{DLou, DLZ, DMZ1, DMZ2}
investigates more general reaction terms. In all these cases, the spreading speed has been established, and 
 sharp estimates of the spreading profile and speed can be found in \cite{DMZ1, DMZ2}. In contrast, the research for systems with free boundaries has been less advanced, due to the extra difficulties arising in the system setting. In \cite{DLin1, GW, GW1, WZjdde, Wu}, various two species competition models with free boundaries have been studied, and
\cite{Wjde, WZ12, ZW} have considered two species predator-prey models with free boundaries. However, the question of whether there is a precise asymptotic spreading speed 
has been left open in \cite{DLin1, GW, GW1, Wjde, WZ12, WZjdde, Wu, ZW}. 

Before stating our results, we first employ the standard change of variables to reduce \eqref{f1} to a simpler form.
Set
\[
u(t,r):=\frac{b_1}{a_1}\,U\left(\frac{t}{a_2}, \sqrt{\frac{d_2}{a_2}}\,r\right),\; v(t,r):=\frac{c_2}{a_2}\, V\left(\frac{t}{a_2}, \sqrt{\frac{d_2}{a_2}}\,r\right),\; h(t)=\sqrt{\frac{a_2}{d_2}}\,H\left(\frac{t}{a_2}\right),
\]
\[
a:=\frac{a_1b_2}{a_2b_1},\; b:=\frac{a_2c_1}{a_1c_2},\;  d:=\frac{d_1}{d_2},\; r:=\frac{a_1}{a_2},\; \mu:=\frac{a_1}{b_1d_2}\hat \mu.
\]
Then a simple calculation shows that \eqref{f1} is equivalent to
\begin{align}
\left\{
\begin{array}{lll}
u_{t}-d \Delta u=ru(1- u-b v),\; & t>0, \ 0\leq r<h(t),  \\
v_{t}- \Delta v=v(1-v-a u),\; & t>0, \ 0\leq r<\infty,  \\
u_r(t,0)=v_r(t,0)=0,\; u(t, r)=0,\quad & t>0,\ h(t)\leq r<\infty,\\
h'(t)=-\mu u_r(t,h(t)),\quad &t>0,\\
h(0)=h_0,  u(0, r)=u_{0}(r),\;  &0\leq r\leq h_0,\\
 v(0, r)=v_0(r), &0\leq r<\infty,\\
\end{array} \right.
\label{p1}
\end{align}
with
\[
h_0:=\sqrt{\frac{a_2}{d_2}}\,H_0,\; u_0(r)=\frac{b_1}{a_1}\,U_0\left( \sqrt{\frac{d_2}{a_2}}\,r\right),\; v_0(r)=\frac{c_2}{a_2}\,V_0\left( \sqrt{\frac{d_2}{a_2}}\,r\right).
\]
 Let us note that under these transformations, \eqref{inf} is equivalent to
\[
a<1<b,
\]
and 
\eqref{sup} is equivalent to 
\[
a>1>b.
\]
Moreover, the spreading-vanishing dichotomy of \cite{DLin1} shows that, for the case $a>1>b$,
either
\begin{itemize}
\item[(i)] ({\it Spreading of $u$}).\ \ $\lim_{t\to+\infty}h(t)=+\infty$ and $\lim_{t\to
+\infty}\big(u(t, r), v(t,r)\big)=\left(1, \, 0\right)$ uniformly in any compact subset of $[0, \infty)$; or
\end{itemize}
\begin{itemize}
\item[(ii)] ({\it Vanishing of $u$}).\ \ $\lim_{t\to+\infty}h(t) <+\infty$\ and \ $\lim_{t\to +\infty}\|u(t,
\cdot)\|_{C([0,h(t)])}=0$,\\ $\lim_{t\to\infty}v(t,r)=1$ uniformly in  any compact subset of $[0,\infty)$.
\end{itemize}

From now on, we will only consider the simplified version \eqref{p1}. The main result of this paper is the following theorem.

\begin{theo}\label{main}
Suppose that $a>1>b$,
\bes\lbl{v_0}
\liminf_{r\to+\infty} v_0(r)>0,
\ees
 and spreading of $u$ happens for \eqref{p1}. Then there exists a unique $s_\mu>0$ such that
\[
\lim_{t\to+\infty}\frac{h(t)}{t}=s_\mu.
\]
Moreover, $s_\mu$ is strictly increasing in $\mu$ and
\[s_0:=\lim_{\mu\to+\infty}s_\mu<+\infty.
\]
\end{theo}

We will show that $s_0$ is the minimal speed of the traveling waves for competition systems considered by Kan-on  in \cite{Kan}.
More precisely, 
 by a suitable transformation, we can apply the result of \cite[Theorem 2.1]{Kan} 
to the  problem 
 \bes\left\{\begin{array}{ll}
 s\Phi'-\Phi''=\Phi(1-\Phi-a\Psi),\; \Phi'<0,\ \ &\xi\in\mathbb{R},\\[1mm]
 s\Psi'-d\Psi''=r\Psi(1-\Psi-b\Phi),\; \Psi'>0, \ \ &\xi\in\mathbb{R},\\[1mm]
(\Phi,\Psi)(-\infty)=(1,0), \ \ (\Phi,\Psi)(\infty)=(0,1)
\end{array}\right.\lbl{3.1}\ees
and obtain the following result.

\begin{prop}
\lbl{kan-on} Suppose that $a>1>b$. Then there exists a constant 
\[
s_0\in \left[2\sqrt{rd(1-b)}, 2\sqrt{rd}\,\right]
\]
 such that  problem \eqref{3.1}
has a  solution when $s\geq s_0$ and it has no solution when $s<s_0$. 
\end{prop}
The number $s_0$ is called the {\it minimal speed} of (\ref{3.1}).
\begin{proof}
If we define
\[
u(\xi):=\Psi\left(\sqrt{\frac dr}\,\xi\right),\; v(\xi):= \frac 1r \Phi\left(\sqrt{\frac dr}\,\xi\right),
\]
and
\[
\tilde a:=\frac 1r,\; \tilde b:=\frac ar,\; \tilde c:=br,\; \tilde d:=\frac 1d,\; \tilde s:=-\frac{s}{\sqrt{dr}},
\]
then a direct computation shows that \eqref{3.1} is equivalent to
\bes\left\{\begin{array}{ll}
u''+ \tilde s u'+ u(1-u-\tilde c v)=0, \; u'>0, &\xi\in\mathbb{R},\\[1mm]
\tilde d v''+ \tilde s v' +v(\tilde a- \tilde b u-v)=0,\; v'<0, &\xi\in\mathbb{R},\\[1mm]
(u,v)(-\infty)=(0,\tilde a), \ \ (u,v)(+\infty)=(1,0).
\end{array}\right.\lbl{3.2}\ees

Clearly
\[
a>1>b \mbox{ is equivalent to } \tilde a<\min\{\tilde b, 1/\tilde c\}.
\]
Hence the conclusions follow directly from Theorem 2.1 of \cite{Kan}, and the strong maximum principle for cooperative systems. We note that the original conclusions in \cite{Kan} do not include $u'>0$ and $v'<0$;
they  only state  that $u$ is nondecreasing and $v$ is nonincreasing. However, by the strong maximum principle applied to the system satisfied by the pair $(u', -v')$, one immediately obtains  $u'>0$ and $v'<0$.
\end{proof}

The value $s_\mu$ in Theorem \ref{main} is determined in the following result.

\begin{theo}\lbl{th3.1}\ Assume that $a>1>b$, and $s_0$ is given in Proposition \ref{kan-on}. Then  the problem
 \bes\left\{\begin{array}{ll}
 s\varphi'-\varphi''=\varphi(1-\varphi-a\psi)\ \ \ \ \ \ (\forall \xi\in\mathbb{R}),\\[1mm]
 s\psi'-d\psi''=r\psi(1-\psi-b\varphi)\ \ \ (\forall \xi>0),\\[1mm]
\varphi(-\infty)=1,\; \varphi'<0\  (\forall\xi\in\mathbb R),\;  \vp(+\infty)=0,\\[1mm]
\psi\equiv 0 \ (\forall \xi\leq 0),\  \psi'>0 \  (\forall \xi\geq 0), \ \psi(+\infty)=1,
 \end{array}\right.\lbl{3.2}\ees
has a unique solution $(\varphi,\psi)\in C^2(\mathbb R)\times \big[C(\mathbb R)\cap C^2([0, \infty))\big]$ for each $s\in[0,s_0)$, and it has no such solution for $s\geq s_0$. 
Furthermore,  if we denote the unique solution by $(\varphi_s,\psi_s)$  $( s\in[0,s_0))$, then the following conclusions hold.
\begin{itemize}
\item[(i)] If $0\le s_1<s_2<s_0$, then 
\[
\psi_{s_1}'(0)>\psi_{s_2}'(0),  \; \psi_{s_1}(\xi)>\psi_{s_2}(\xi) \; (\forall \xi>0), \; \vp_{s_1}(\xi)<\vp_{s_2}(\xi) \; (\forall \xi\in\R).
\]
\item[(ii)] The operator $s\mapsto (\vp_s, \psi_s)$ is continuous from $[0,s_0)$ to $C_{loc}^2(\mathbb R)\times C^2_{loc}([0,+\infty))$.
Moreover, 
\[
\mbox{$\lim_{s\to s_0}(\vp_s,\psi_s)=(1, 0)$ in $C_{loc}^2(\mathbb R)\times C^2_{loc}([0,+\infty))$.}
\]
\item[(iii)] For any $\mu>0$, there exists a unique $s=s_\mu\in (0, s_0)$ such that
\[
\mu \psi'_{s_\mu}(0)=s_\mu. 
\]
Moreover,
\[
\mu\mapsto s_\mu \mbox{ is strictly increasing and } \lim_{\mu\to+\infty}s_\mu=s_0.
\]
\end{itemize}
 \end{theo}

For each $s\in [0, s_0)$, the solution pair $(\vp_s,\psi_s)$ in Theorem \ref{th3.1} generates a  traveling wave 
\[
(\tilde u(t,x), \tilde v(t,x)):=(\psi_s(st-x), \vp_s(st-x)),
\]
which satisfies
\begin{align}
\left\{
\begin{array}{lll}
\tilde u_{t}-d \tilde u_{xx}=r\tilde u(1- \tilde u-b\tilde v),\; & t>0, \ -\infty< x<st,  \\
\tilde v_{t}- \tilde v_{xx}=\tilde v(1-\tilde v-a \tilde u),\; & t>0, \ x\in\R^1,  \\
\tilde u(t, x)=0,\quad & t>0,\ st\leq x<+\infty.
\end{array} \right.
\label{p2}
\end{align}
We note that when $s=s_\mu$, one has the extra identity
\[
s=-\mu \tilde u_x(t, st).
\]
We will call $(\vp_{s_\mu}, \psi_{s_\mu})$ the semi-wave associated to \eqref{p1}.

The methods developed in this paper can  be extended to treat the strong competition case (namely the case $\min\{a, b\}>1$),
which will be considered in a separate paper. Without the assumption \eqref{v_0}, the conclusion of Theorem \ref{main} need not be true; see Remark \ref{v_0-not} for details. Note also that the proof of Theorem \ref{main} provides estimates on the profile of $u(t,x)$ and $v(t,x)$ for large $t$, showing that they behave like a pair of semi-wave; see Remark \ref{remark2}.

It is interesting to note that,  the approach in this paper, namely making use of the semi-wave
 to determine the spreading speed of the free boundary model \eqref{p1}, is rather different from the approaches used for spreading governed by the Cauchy problem of two species reaction diffusion systems, where the spreading speed  is usually determined by very different methods. We list two examples below.

In \cite{LLW}, the spreading speed of an invasive species into the territory of an existing competitor, determined by the corresponding Cauchy problem of \eqref{p1} in one space dimension, is established by an approach  developed in \cite{WLL}
(see also references therein),
which does not depend on the existence of a corresponding traveling wave. 

In \cite{D}, the invasion of a predator to the territory of a prey species was investigated,
where the interaction and growth of the species are governed by a Holling-Tanner type predator-prey system of the form
\begin{equation}
\label{pp}
\left\{
\begin{array}{ll}
u_t-d\Delta u=u(1-u)-\Pi(u)v, & x\in\R^N, \; t>0,\\
v_t-\Delta v=rv\left(1-\frac vu\right), & x\in\R^N, \; t>0,\\
u(0,x)=u_0(x),\; v(0,x)=v_0(x), & x\in\R^N,
\end{array}
\right.
\end{equation}
with
\[
1\geq u_0(x)\geq \delta>0,\; 1\geq v_0(x)\geq 0 \; (\forall x\in\R^N).
\]
Moreover, as usual, it is assumed that $v_0$ has nonempty compact support. Under suitable conditions for the function $\Pi(u)$, it is shown in
\cite{D} that the predator species $v$ invades at the asymptotic speed $c^*=2\sqrt{r}$. In space dimension $N=1$,
more accurate estimate of the spreading speed is obtained, and the existence of an almost planar ``generalized transition wave'' (see \cite{BH}) is established. Again, these results are proved without knowing the existence of a corresponding traveling wave. Indeed, whether the generalized transition wave mentioned above is actually a traveling wave is still
an open problem.

The rest of this paper is organized as follows. In Section 2, we prove Theorem \ref{th3.1},
which relies on Proposition \ref{kan-on}, some subtle constructions of comparison functions, and
a sliding method. Section 3 gives the proof of Theorem \ref{main}, based on Theorem \ref{th3.1}
and suitable comparison arguments.

\section{Semi-wave solutions}
\setcounter{equation}{0}
This section constitutes the proof  of Theorem \ref{th3.1}. We will accomplish this by a series of lemmas.
The first is  a comparison principle.

\begin{lem}{\rm(}Comparison principle{\rm)}\lbl{p3.1}\, Let $f, g\in C([0,+\infty))$  with $g$ nonnegative and not identically 0, and $s,d$ be given constants with $d>0$. Assume that $u_i\in C^2([0,\infty))$, and satisfies $u_i(x)>0$ in $(0,\infty)$, $u_1(0)\leq u_2(0)$ and
 \[su_1'-du_1''-u_1\big[f(x)-g(x)u_1\big]\leq 0\leq  su_2'-du_2''-u_2\big[f(x)-g(x)u_2\big] \ \ (\forall x>0).\]
If $\dd\limsup_{x\to\infty}\frac{u_1(x)}{u_2(x)}\leq 1$, then
  \[u_1(x)\leq u_2(x) \ \ ( \forall  x\geq 0).\]
 \end{lem}
\begin{proof}
For any small constant $\delta>0$, $(1+\delta)u_2$ satisfies the same differential inequality as $u_2$, and we may apply Lemma 2.1 of \cite{DMa} over $[0, L_\delta]$
to deduce that $u_1(x)\leq (1+\delta)u_2(x)$ for $x\in [0, L_\delta]$,
where $L_\delta>0$ satisfies
\[
(1+\delta)u_2(L_\delta)>u_1(L_\delta),\; \lim_{\delta\to 0} L_\delta=+\infty.
\]
 The required inequality then follows by letting $\delta\to 0$. 
\end{proof}

To motivate our definition of $s^*$ below, we note that if
 $s\geq 0$ and $(\varphi,\psi)$ is a solution of (\ref{3.2}), then
 \bess
 &\dd F_1(\varphi, \psi)(\xi):=\frac{\varphi(\xi)[1-\varphi(\xi)-a\psi(\xi)]+\varphi''(\xi)}{\varphi'(\xi)}=s \ \ (\forall  \xi\in\mathbb{R}),&\\
 &\dd F_2(\varphi, \psi)(\xi):=\frac{r\psi(\xi)[1-\psi(\xi)-b\varphi(\xi)]+d\psi''(\xi)}{\psi'(\xi)}=s \ \ (\forall  \xi>0).&
\eess

We now use $K_1$ to denote the set of functions $\varphi\in C^2(\mathbb{R})$ satisfying
 \[
\mbox{ $\varphi(-\infty)=1$,  $\varphi'(\xi)<0 \; (\forall \xi\in\mathbb{R})$,  $\varphi(+\infty)= 0$,}
\]
and let $K_2$ denote the set of functions $\psi\in C(\mathbb{R})\cap C^2([0,\infty))$ such that 
\[
\mbox{
$\psi(\xi)\equiv 0 \; (\forall \xi\leq 0)$,  $\psi'(\xi)>0 \; (\forall \xi\geq 0)$,  $\psi(+\infty)= 1$.}
\]
 Define
 \[K=\{(\varphi,\psi):\, \varphi\in K_1,\,\psi\in K_2\}.\]
For $(\varphi,\psi)\in K$, we set
 \bess
 F(\varphi,\psi)=\min\left\{\inf_{\xi\in\mathbb{R}}
 F_1(\varphi, \psi)(\xi),\
 \inf_{\xi>0}F_2(\varphi, \psi)(\xi)\right\}.
 \eess

Clearly,  if $(\varphi,\psi)$ is a  solution of (\ref{3.2}) with $s\ge 0$, then $(\varphi,\psi)\in K$ and $F(\varphi,\psi)=s$. Therefore
 \[s\le s^*:=\sup_{(\varphi,\psi)\in K}F(\varphi,\psi).\]
We can thus conclude that, if $s^*$  is finite, then
  \eqref{3.2} has no  solution when $s>s^*$. From the definition of $s^*$, we can use a super- and sub-solution argument to show that \eqref{3.2} has a solution for every $s\in [0, s^*)$.
We will also show that $s^*=s_0$, and \eqref{3.2} has no solution for $s\geq s_0$. Throughout the remainder of this section, we will always assume
\[
a>1>b.
\]
We first prove
\begin{lem}\lbl{s*>} 
$s^*\geq s_0$.
\end{lem}

Let $(\Phi_0,\Psi_0)$ be a solution of \eqref{3.1} with $s=s_0$.
The proof of Lemma \ref{s*>}  depends on an asymptotic analysis of $\Psi_0$.
 \begin{lem} \lbl{(3.1)s_0}
 $$ 
\lim_{x\to-\infty}\frac{\Psi_0'(x)}{\Psi_0(x)}= \beta_1:=\frac{s_0+\sqrt{s_0^2-4rd(1-b)}}{2d}.
$$
Moreover, as $x\to-\infty$, 
\begin{itemize}
\item[(i)] in the case $s_0>2\sqrt{rd(1-b)}$,
\[
\Psi_0(x)=c_0e^{\beta_1 x}(1+o(1)) \mbox{ for some $c_0>0$ },
\]
\item[(ii)] in the case $s_0=2\sqrt{rd(1-b)}$,
\[
\Psi_0(x)=c_0e^{\beta_1 x}(1+o(1)) \mbox{ or } \Psi_0(x)=c_0|x|e^{\beta_1 x}(1+o(1))\mbox{ for some $c_0>0$ }.
\]
\end{itemize}
 \end{lem}
\begin{proof}
A simple calculation indicates that the ODE system satisfied by $(\Phi_0, \Phi_0', \Psi_0, \Psi_0')$ has $(1,0,0,0)$ as a critical point, which is a saddle point. It follows from standard theory on stable and unstable manifolds 
(see, e.g., Theorem 4.1 and its proof in Chapter 13 of \cite{CL}), that $1-\Phi_0(x)$ and $\Psi_0(x)$
converge to 0 exponentially as $x\to-\infty$.

  We may rewrite the equation satisfied by $\Psi_0$ in the form
\bes\lbl{Psi0}
d\Psi_0''-s_0\Psi_0'+r(1-b)\Psi_0+\epsilon(x)\Psi_0=0,
\ees
with
\[
\epsilon(x):=rb(1-\Phi_0(x))-r\Psi_0(x)\to 0 \mbox{ exponentially as } x\to-\infty.
\]
Clearly
\[ u_1:=e^{\beta_1 x},\; u_2:=\left\{\begin{array}{ll} e^{\beta_2 x}, & s_0>2\sqrt{rd(1-b)},\\[1mm]
-xe^{\beta_1 x}, & s_0=2\sqrt{rd(1-b)},
\end{array}\right.
\]
are linearly independent solutions of the linear equation
\[
du''-s_0u'+r(1-b)u=0,
\]
where 
\[
 \beta_2:=\frac{s_0-\sqrt{s_0^2-4rd(1-b)}}{2d}\leq \beta_1.
\]

We are now in a position to apply  to  \eqref{Psi0} Theorem 8.1 in Chapter 3 of \cite{CL} (for the case $\beta_1\not=\beta_2$), or a variant of this result (see Question 35 in Chapter 3 of \cite{CL} or Theorem 13.1 in Chapter X of \cite{H}, for the case $\beta_1=\beta_2$), 
to conclude that \eqref{Psi0} has two linearly independent solutions ${\tilde u_i}$, $i=1,2$, satisfying
\[
{\tilde u_i}(x)=(1+o(1)){u_i}(x), \; \tilde u_i'(x)=(1+o(1))u_i'(x)\; \mbox{ as } x\to-\infty,\; i=1,2.
\]
Since $\Psi_0$ solves \eqref{Psi0}, there exist constants $c_1$ and $c_2$ such that
\[
\Psi_0(x)=c_1\tilde u_1(x)+c_2\tilde u_2(x).
\]
From $\Psi_0(x)>0$ and $\beta_1\geq \beta_2>0$ we deduce that
either 
\begin{center}
(i) $c_2>0$ or (ii) $c_2=0$ and $c_1>0$.
\end{center}

If $s_0>2\sqrt{rd(1-b)}$, according to the proof of Theorem 2.1 in \cite{Kan} (see page 163 there), we have 
$$
\limsup _{x\rightarrow -\infty}\{\Psi_0(x)e^{-\beta_1x}\}<+\infty,
$$
which implies that, in such a case, $c_2=0$ and $c_1>0$. We can thus conclude that, 
in the case $ s_0>2\sqrt{rd(1-b)}$, there exists $c_0>0$ such that
\bes\lbl{beta1>beta2}
\Psi_0(x)=c_0u_1(x)(1+o(1)),\;  \Psi_0'(x)=c_0u_1'(x)(1+o(1)) \; \mbox{ as } x\to-\infty,
\ees
and  in the case $ s_0=2\sqrt{rd(1-b)}$, there exists $c_0>0$ so that, as $x\to-\infty$,
\bes\lbl{beta1=beta2-1}
\Psi_0(x)=c_0 u_1(x)(1+o(1)),\; \Psi_0'(x)=c_0 u_1'(x)(1+o(1)),\;\; \mbox{ or}
\ees
\bes\lbl{beta1=beta2-2}
 \Psi_0(x)=c_0 u_2(x)(1+o(1)),\; \Psi_0'(x)=c_0 u_2'(x)(1+o(1)).
\ees
The conclusions of the lemma are direct consequences of \eqref{beta1>beta2}, \eqref{beta1=beta2-1} and \eqref{beta1=beta2-2}.
\end{proof}

{\bf Proof of Lemma \ref{s*>}:}  For  arbitrarily small $\epsilon>0$, we are going to construct a function pair $(\bar{\varphi},\underline{\psi})\in K$ such that
\bes\lbl{phi-psi}
 \left\{\begin{array}{lll}
(s_0-\epsilon)\bar{\varphi}'-\bar{\varphi}''\geq\bar{\varphi}(1-\bar{\varphi}-a\underline{\psi}), & \xi\in \mathbb{R},\\[1mm]
(s_0-\epsilon)\underline{\psi}'-d\underline{\psi}''\leq r\underline{\psi}(1-\underline{\psi}-b\bar{\varphi}), & \xi\in (0, \infty).
 \end{array}\right.\ees
Clearly this would imply $s^*\geq F(\bar\varphi,\underline\psi)\geq s_0-\epsilon$. Since $\epsilon>0$ is arbitrary, the required estimate $s^*\geq s_0$ then follows.

Fix  $\epsilon\in (0, rb/\beta_1)$, where, as before,  $\beta_1=\frac{s_0+\sqrt{s_0^2-4rd(1-b)}}{2d}$.
In view of the asymptotic behavior  of $(\Phi_0, \Psi_0)$ and Lemma \ref{(3.1)s_0}, there exists a constant $M_0<0$ such that
\bes\lbl{M0}
\frac{\Psi_0'(x)}{\Psi_0(x)}\geq \frac{1}{2}\beta_1,\; \max\big\{1-\Phi_0(x),\Psi_0(x)\big\}<\min\left\{\frac{\epsilon\beta_1}{4rb}, 1-b\right\} \; (\forall x \leq M_0).
\ees

{\bf Step 1:} 
We construct $(\tilde{\varphi},\tilde{\psi})$ that satisfies,  in the weak sense,
\bes\lbl{tilde-phi}
(s_0-\epsilon)\tilde{\varphi}'-\tilde{\varphi}''\geq\tilde{\varphi}(1-\tilde{\varphi}-a\tilde{\psi}) \mbox{ for } x\in \mathbb{R},
\ees
\bes\lbl{tilde-psi}
(s_0-\epsilon)\tilde\psi'-d\tilde\psi''\leq r\tilde\psi(1-\tilde\psi-b\tilde\varphi) \mbox{ for $x$ satisfying } \tilde\psi(x)>0,
\ees
 and moreover,
\[\mbox{
$\tilde \varphi(x)$ is nonincreasing, $\tilde \psi(x)$ is nondecreasing,}
\]
\[
\tilde\varphi(+\infty)=0,\; \tilde\psi(+\infty)=1,\; \tilde\varphi(-\infty)=1,\;  \tilde\psi(x)=0 \mbox{ for all large negative } x.
\]
 The required $(\bar\phi, \underline\psi)$ will be obtained in Step 2, by  
solving two natural parabolic problems with $\tilde \varphi$ and $\tilde \psi$ as initial functions, respectively.

Let $M_0$ be given in \eqref{M0}.
We temporarily define $\tilde\varphi:=\Phi_0+\epsilon_1p$, with $p(x)$ determined as follows.
\[
p(x):=e^{\frac{s_0-\sqrt{s_0^2+2}}{2}x} \mbox{ for } x\leq M_0-1;\;\; p(x):=0 \mbox{ for } x\geq M_0,
\]
and for $x\in (M_0-1, M_0)$, $p(x)>0$, $p'(x)\leq 0$. Moreover, $p(x)$ is $C^2$ everywhere. The positive constant $\epsilon_1$ will be determined below. Later on, we will modify $\tilde\varphi(x)$ for large negative $x$.

We now calculate
\bes\lbl{tilde-phi-1}
\begin{array}{l}
(s_0-\epsilon)\tilde\varphi'-\tilde \varphi''-\tilde\varphi (1-\tilde\varphi-a\Psi_0)\\[1mm]
=-\epsilon\Phi_0'+\epsilon_1\big[(s_0-\epsilon) p'-p''+p(-1+2\Phi_0+\epsilon_1p+a\Psi_0)\big].
\end{array}
\ees
Hence we can fix $\epsilon_1>0$ sufficiently small so that, for $x\in [M_0-1, M_0]$,
\bes\lbl{[M0-1,M0]}
(s_0-\epsilon)\tilde\varphi'-\tilde \varphi''-\tilde\varphi (1-\tilde\varphi-a\Psi_0)>0
\ees
and
\[
\tilde \varphi'(x)=\Phi_0'(x)+\epsilon_1p'(x)<0,\; \; \tilde\varphi(x)=\Phi_0(x)+\epsilon_1 p(x)<1.
\]
By the definition of $p(x)$ for $x\leq M_0-1$, clearly $\tilde \varphi'(x)<0$ for $x\leq M_0-1$, and
$\tilde\varphi (x)\to+\infty$ as $x\to-\infty$.  Hence there exists a unique  constant $M_1<M_0-1$ such that 
\[
\tilde \varphi(M_1)=\Phi_0(M_1)+\epsilon_1 p(M_1)=1.
\]
In view of \eqref{M0}, we have
\bes\lbl{phi-p}
 \epsilon_1 p(M_1)=1-\varphi(M_1)<\frac{\epsilon\beta_1}{4rb} \mbox{ and hence }   \epsilon_1 p(x)< \frac{\epsilon\beta_1}{4rb} \mbox{ for } x\in [M_1, M_0].
\ees
Let $\epsilon_1$ and $M_1$ be chosen as above. 
We define
\bes \tilde{\varphi}(x):=
 \left\{\begin{array}{lll}
 \Phi_0(x)+\epsilon_1p(x), & x\geq M_1,\\
 1, & x<M_1.
  \end{array}\right.\ees
Clearly $0<\tilde \psi(x)\leq 1$ for all $x$.
We also define 
\[
\tilde\psi(x)=\Psi_0(x) \mbox{ for }x\geq M_1,
\]
and suppose $\tilde \psi(x)\geq 0$ for $x<M_1$ (with the exact definition of $\tilde \psi$ in this range to be specified later).

We show next that $(\tilde\varphi,\tilde \psi)$ satisfies  \eqref{tilde-phi}. For $x\geq M_0$,
this is obvious, since $(\tilde\varphi,\tilde\psi)=(\Phi_0,\Psi_0)$ in this range. For $x\in [M_0-1, M_0]$,
this has been proved in \eqref{[M0-1,M0]}. For $x<M_1$, it also holds trivially since $\tilde\psi\geq 0$. We now examine it for $x\in [M_1, M_0-1]$. Firstly, we note that $\tilde\varphi$ is $C^2$ except a jumping discountinuity of $\tilde\varphi'(x)$ at $x=M_1$, where we have
\[
\tilde\varphi'(M_1-0)=0> \tilde\varphi'(M_1+0),
\]
which is the right inequality for the required differential inequality in the weak sense. For $x\in (M_1, M_0-1)$,
by the choice of $M_0$, we have $\Phi_0(x)>1-\frac{\epsilon\beta_1}{4rb}>\frac34$.
Hence, for such $x$,
\[
(s_0-\epsilon) p'-p''+p(-1+2\Phi_0+\epsilon_1p+a\Psi_0)>s_0p'-p''+\frac12 p=0.
\]
Therefore we can apply \eqref{tilde-phi-1} to deduce
\[
(s_0-\epsilon)\tilde\varphi'-\tilde \varphi''-\tilde\varphi (1-\tilde\varphi-a\Psi_0)>0 \mbox{ for } x\in (M_1, M_0-1].
\]

We have thus varified that $(\tilde\varphi,\tilde\psi)$ satisfies  (in the weak sense)  \eqref{tilde-phi}.
Moreover, from the definition, we also see that $\tilde\varphi$ is nonincreasing.

Next, we show that $\tilde\psi(x)$ can be suitably defined for $x<M_1$ such that \eqref{tilde-psi} is satisfied.
For $x\geq M_0$, this inequality follows from the fact that $(\tilde\varphi,\tilde\psi)=(\Phi_0,\Psi_0)$. For $x\in (M_1, M_0]$, due to \eqref{phi-p} and \eqref{M0},
we have 
\[
-\frac\epsilon 2 \beta_1+rb\epsilon_1 p(x)\leq -\frac{\epsilon\beta_1}{4}<0
\]
and 
\[
\begin{array}{l}
(s_0-\epsilon)\Psi_0'-d\Psi_0''-r\Psi_0(1-\Psi_0-b\tilde\varphi)\medskip\\
\displaystyle =-\epsilon\Psi_0'+r\Psi_0 b\epsilon_1 p
\leq \left[-\frac\epsilon 2 \beta_1+rb\epsilon_1 p(x)\right]\Psi_0<0.
\end{array}
\]
Thus  \eqref{tilde-psi} is satisfied by $(\tilde\varphi,\tilde\psi)$ for $x>M_1$.

Next we define $\tilde\psi(x)$ for $x<M_1$ so that \eqref{tilde-psi} is satisfied by $(\tilde\varphi,\tilde\psi)$
for $x\leq M_1$. We will treat the cases $s_0>2\sqrt{rd(1-b)}$ and $s_0=2\sqrt{rd(1-b)}$
separately.

{\it Case 1}: $s_0>2\sqrt{rd(1-b)}$. 

 In this case, we assume further that $\epsilon>0$ is sufficiently small so that
\[
\beta_\epsilon:=\frac{(s_0-\epsilon)+\sqrt{(s_0-\epsilon)^2-4dr(1-b)}}{2d} >0.
\]
We are going to choose a constant $\epsilon_2>0$ and a function $q(x)$ such that 
\[
\tilde\psi(x):=\Psi_0(x)-\epsilon_2 q(x)
\]
 meets all the requirements. 

We define 
\[
\mbox{$q(x):=0$ for $x\geq M_1$}, \;
q(x):=e^{\beta_\epsilon x} \mbox{ for }
x\in (-\infty, M_1-1),
\]
and for $x\in [M_1-1, M_1]$,  we define $q(x)$ so that $q(x)>0$, and $q(x)$ is $C^2$ everywhere.

Since $\beta_\epsilon<\beta_1$, by Lemma \ref{(3.1)s_0} we can find $M_\epsilon<M_1-1$ such that
\[
\Psi_0'(x)>\beta_\epsilon \Psi_0(x) \mbox{ for } x\leq M_\epsilon.
\]
It follows that, for $x\leq M_\epsilon$,
\bes\lbl{tilde-psi'}
\tilde\psi'(x)=\Psi_0'(x)-\epsilon_2q'(x)>\beta_\epsilon \Psi_0(x)-\epsilon_2 \beta_\epsilon q(x)=\beta_\epsilon \tilde \psi(x).
\ees
We now fix $\epsilon_2$ sufficiently small such that, for $x\in [M_\epsilon, M_1]$,
\[
\tilde\psi(x)>0,\; \tilde\psi'(x)>0,
\]
and
\[
\begin{array}{l}
(s_0-\epsilon)\tilde\psi'-d\tilde\psi''-r\tilde\psi(1-\tilde \psi-b\tilde\varphi)\\[1mm]
=rb\Psi_0(1-\Phi_0)-\epsilon \Psi_0'+\epsilon_2\big[(\epsilon-s_0)q'+dq''+rq(1-2\Psi_0+\epsilon_2q-b)\big]\\[1mm]
\leq -\frac{\epsilon}{4}\beta_1\Psi_0+\epsilon_2\big[(\epsilon-s_0)q'+dq''+rq(1-b)\big]<0.
\end{array}
\]
Here in deriving the second last inequality, we have used \eqref{M0} and
\[
 -2\Psi_0+\epsilon_2q<-\tilde\psi<0.
\]
Thus \eqref{tilde-psi} is satisfied by $(\tilde \varphi,\tilde\psi)$ for $x\geq M_\epsilon$.

Due to $\beta_\epsilon<\beta_1$, from Lemma \ref{(3.1)s_0}  we easily deduce
\[
\lim_{x\to-\infty}\frac{\Psi_0(x)}{e^{\beta_\epsilon x}}=0.
\]
It follows that 
\[
\tilde\psi(x)=e^{\beta_\epsilon x}\left[\frac{\Psi_0(x)}{e^{\beta_\epsilon x}}-\epsilon_2\right]<0
\mbox{ for all large negative } x.
\]
Since $\tilde\psi(M_\epsilon)>0$, by continuity there exists $M_2<M_\epsilon$ such that
\[
\tilde\psi(M_2)=0,\;\; \tilde\psi(x)>0 \mbox{ for } x\in (M_2, M_\epsilon],
\]
which implies
$-2\psi_0+\epsilon_2q<-2\tilde\psi<0$ for such $x$.
By the definition of $q(x)$, we have
\[
(\epsilon-s_0)q'+dq''+rq(1-b)=0 \mbox{ for } x<M_1-1.
\]
Thus for $x\in (M_2, M_\epsilon]$, we have
\[
\begin{array}{l}
(s_0-\epsilon)\tilde\psi'-d\tilde\psi''-r\tilde\psi(1-\tilde \psi-b\tilde\varphi)\\[1mm]
\leq -\frac{\epsilon}{4}\beta_1\Psi_0+\epsilon_2\big[(\epsilon-s_0)q'+dq''+rq(1-b)\big]\\[1mm]
= -\frac{\epsilon}{4}\beta_1\Psi_0<0.
\end{array}
\]
We have thus proved that  \eqref{tilde-psi} is satisfied by $(\tilde \varphi,\tilde\psi)$ for $x>M_2$. Moreover,
\[ \tilde\psi(x)>0 \mbox{ for } x>M_2, \; \tilde \psi(M_2)=0.\]
By the choice of $\epsilon_2$ and the definition of $\tilde\psi$, we already know that $\tilde\psi'(x)>0$ for
$x\geq M_\epsilon$.
By \eqref{tilde-psi'}, we deduce  $\tilde\psi'(x)>0$ for $x\in (M_2, M_\epsilon]$. Hence 
\[
\tilde\psi'(x)>0 \mbox{ for } x>M_2.
\]
We may now define
\[
\tilde\psi(x)=0 \mbox{ for } x\leq M_2,
\]
and conclude that $(\tilde\varphi, \tilde\psi)$ meets all the requirements of Step 1.

{\it Case 2}: $s_0=2\sqrt{rd(1-b)}$.

In this case we have $s_0-\epsilon<2\sqrt{dr(1-b)}$, and hence we can use Proposition 2.1 of \cite{BDK}
to see that the problem
\[
(s_0-\epsilon)\psi'-d\psi''=r\psi(1-b-\psi) \; \mbox{ for } x>0,\; \psi(0)=0
\]
has a unique positive solution, and it satisfies
\[
\psi(+\infty)=1-b,\; \psi'(+\infty)=0.
\]
(Note that a simple change of variables can transform the above problem to the form considered in \cite{BDK}.)
By \eqref{M0}, we have 
\[
\Psi_0(M_1)<1-b.
\]
Hence we can find a large positive constant $M^0$ such that
\[
\psi(M^0)>\Psi_0(M_1),\; \psi'(M^0)<\Psi_0'(M_1).
\]
Set
\[
m:=\frac{\Psi_0( M_1)}{\psi(M^0)},\;\; M_2:=M_1-M^0,\;\;  \tilde\psi(x)=\left\{\begin{array}{ll} \Psi_0(x),& x\geq M_1,\\
m\psi(x-M_2), & M_2<x<M_1,\\
0, & x\leq M_2.
\end{array}\right.
\]
Clearly $\tilde\psi$ is continuous in $\mathbb{R}$, and
\[
\tilde\psi'(x)>0 \mbox{ for } x\in (M_2, M_1)\cup(M_1, +\infty),\; \tilde\psi'(M_1-0)<\tilde\psi'(M_1+0).
\]
Moreover, for $x\in (M_2, M_1)$, due to $0<m<1$,
we have
\[
(s_0-\epsilon)\tilde\psi'-d\tilde\psi''\leq r\tilde\psi(1-\tilde\psi-b).
\]
Therefore  \eqref{tilde-psi} is satisfied (in the weak sense) by $(\tilde\varphi, \tilde\psi)$ for $x>M_2$.
So in case 2, we have also constructed $(\tilde\varphi,\tilde\psi)$ that meets all the requirements of Step 1.

{\bf Step 2}: Definition of
 $(\bar{\varphi},\underline{\psi})$ and completion of the proof.

Set
\[
\tilde\varphi_0(x)=\tilde\varphi(x+M_2),\;\; \tilde\psi_0(x)=\tilde\psi(x+M_2).
\]
Then consider the following auxiliary problems:

\bes\lbl{para-phi}
 \left\{\begin{array}{lll}
\varphi_t+(s_0-\epsilon)\varphi_x-\varphi_{xx}=\varphi(1-\varphi-a\tilde\psi_0(x)), & t>0,x\in \mathbb{R},\\[1mm]
\varphi(0,x)=\tilde\varphi_0(x),& x\in \mathbb{R},
\end{array}\right.
\ees
\bes\lbl{para-psi}
 \left\{\begin{array}{lll}
\psi_t+(s_0-\epsilon)\psi_x-d\psi_{xx}= r\psi(1-\psi-b\tilde\varphi_0(x)), & t>0, x>0,\\[1mm]
\psi(t,0)=0, & t>0,\\[1mm]
\psi(0,x)=\tilde\psi_0(x),& x\geq 0.
 \end{array}\right.
\ees

From Step 1, we see that $\tilde\varphi_0$ is a strict supersolution of the corresponding elliptic problem of \eqref{para-phi}.
It follows that  $\varphi_t<0$ for $x\in \mathbb{R}$ and $t>0$. Moreover, due to the monotonicity of $\tilde\varphi_0$ and $\tilde\psi_0$,  one may use the strong maximum principle to the equation for
$\varphi_x(t, x)$ to deduce that $\varphi_x(t,x)<0$ for every $t>0$ and $x\in\mathbb R$.

Similarly, making use of \eqref{para-psi} we obtain $\psi_t>0$ for $x>0$ and $t>0$, and $\psi_x(t,x)>0$ for every $t>0$ and $x\geq 0$.

Define
\[
\overline \varphi(x):=\varphi(1,x),\; \underline\psi(x):=\psi(1,x).
\]
Then we have 
\[
\overline \varphi(x)<\tilde\varphi_0(x),\;\; \underline \psi(x)>\tilde\psi_0(x),
\]
and
\[
\varphi_t(1,x)<0,\; \; \ \psi_t(1,x)>0.
\]
Let us  also note that since $1>\tilde\varphi_0>0$ in $\mathbb R$ and $1>\tilde\psi_0>0$ in $(0, \infty)$, for every $t>0$, we have
\[
\varphi(t,x)>0 \mbox{ in } \mathbb R,\;\; 1>\psi(t,x)>0 \mbox{ in } (0,\infty).
\]
Hence it follows from \eqref{para-phi} and \eqref{para-psi}  that
$(\overline \varphi,\underline \psi)$ satisfies \eqref{phi-psi}. Moreover, due to $\tilde\varphi_0(+\infty)=0$
and $\tilde\psi_0(+\infty)=1$, we further deduce
\[
\overline\varphi(+\infty)=0,\; \underline\psi(+\infty)=1.
\]
Finally we show that 
\[
\overline\varphi(-\infty)=1.
\]
Let $u_0(x)$ be the unique positive solution to
\[
-u''=u(1-u) \mbox{ for } x<0,\; \; u(0)=0.
\]
Then $u_0(-\infty)=1$ and $u_0'(x)<0$ for $x\leq 0$. Define
\[
\tilde u_0(x)=\left\{\begin{array}{ll} u_0(x),& x\leq 0,\\
0,& x>0.
\end{array}\right.
\]
Since $\tilde\varphi_0(x)=1$  and $\tilde\psi_0(x)=0$ for $x\leq 0$, we easily see that $\tilde u_0$ is a subsolution of the corresponding elliptic problem of \eqref{para-phi}. It then follows from $\tilde u_0\leq \tilde\varphi_0$ that $\varphi(t,x)\geq \tilde u_0(x)$ for all $t>0$ and $x\in\mathbb R$. Hence we must have $\varphi(t,-\infty)=1$
for all $t>0$. In particular, $\overline\varphi(-\infty)=\varphi(1,-\infty)=1$. Therefore $(\overline\varphi,\underline\psi)\in K$.
This completes the proof. \hfill \fbox

\begin{lem}\lbl{existence}
Problem \eqref{3.2} has a solution for every $s\in[0,s^*)$.
\end{lem}
\begin{proof}
 Taking advantage of the order-preserving property of \eqref{3.2}, we will first construct  super- and subsolutions of \eqref{3.2}, and then use  them to obtain a solution.

{\bf Step 1}.  Construction of $(\underline{\varphi}, \bar\psi)$.

Let $\underline{\varphi}(\xi)$ and $\bar{\psi}(\xi)$ be the unique positive solutions of
 \[-\varphi''=\varphi(1-\varphi) \ (\forall \xi<0), \quad \varphi(0)=0\]
and
 \[-d\psi''=r\psi(1-\psi)\ (\forall \xi>0),\quad \psi(0)=0,\]
respectively. Then
 \[\underline{\varphi}'(\xi)<0 \ (\forall\ \xi\leq 0), \ \ \ \bar{\psi}'(\xi)>0 \ (\forall\ \xi\geq 0), \ \ \ \underline{\varphi}(-\infty)=\bar{\psi}(\infty)=1.\]

We extend $\underline{\varphi}(\xi)$ and $\bar{\psi}(\xi)$ by the value $0$ to $\xi>0$ and $\xi<0$, respectively. 
Since $\underline{\varphi}'(0^-)<0=\underline{\varphi}'(0^+)$, for each $s\geq 0$, we have (in the weak sense for $\underline\varphi$),
 \bess\left\{\begin{array}{ll}
 s\underline{\varphi}'-\underline{\varphi}''\leq \underline{\varphi}(1-\underline{\varphi}-a\bar{\psi}), \ \ &\xi\in\mathbb R,\\[2mm]

 \underline{\varphi}(-\infty)=1, \ \ \ \underline{\varphi}(\infty)=0,
 \end{array}\right.\eess
and
 \bess\left\{\begin{array}{ll}
 s\bar{\psi}'-d\bar{\psi}''\geq r\bar{\psi}(1-\bar{\psi}-b\underline{\varphi}), \ \ \xi>0,\\[2mm]
 \bar{\psi}(0)=0, \ \ \bar{\psi}(\infty)=1.
 \end{array}\right.\eess

{\bf Step 2}. Construction of $(\bar{\varphi},\underline{\psi})$ satisfying 
\bes\lbl{3.4}
\mbox{$\bar\varphi\geq \underline\varphi$, $\underline\psi\leq\bar\psi$.}
\ees

For $s\in[0,s^*)$, by the definition of $s^*$, there exists  $(\bar\varphi,\underline{\psi})\in K$ such that $F(\bar{\varphi},\underline{\psi})>s$. Consequently,
 %\bess
 %\bar{\varphi}(\xi)[1-\bar{\varphi}(\xi)-a\underline{\psi}(\xi)]
 %+\bar{\varphi}''(\xi)<s\bar{\varphi}'(\xi), \ \ \forall \ \xi\in\mathbb{R},\\
 %r\underline{\psi}(\xi)[1-\underline{\psi}(\xi)-b\bar{\varphi}(\xi)]
 %+d\underline{\psi}''(\xi)>s\underline{\psi}'(\xi), \ \ \forall \ %\xi\geq 0.\eess
 %Thus we have
\bess\left\{\begin{array}{ll}
  s \bar{\varphi}'-\bar{\varphi}''> \bar{\varphi}(1-\bar{\varphi}-a\underline{\psi}), \qquad \bar{\varphi}'<0, \ \ \ \xi\in\mathbb{R},\\[1mm]
 s \underline{\psi}'-d\underline{\psi}''< r\underline{\psi}(1-\underline{\psi}-b\bar{\varphi}),\ \ \ \underline{\psi}'>0, \ \ \ \xi\geq 0,\\[1mm]
 \bar{\varphi}(-\infty)=1, \;  \bar{\varphi}(\infty)=0,\ \ \underline{\psi}(0)=0, \ \ \underline{\psi}(\infty)= 1.
 \end{array}\right.\eess

Now we prove (\ref{3.4}). To prove $\underline{\varphi}(\xi)\leq \bar{\varphi}(\xi)$ in $\mathbb{R}$, it is enough to show that this is true for $\xi<0$ since $\underline{\varphi}(\xi)\equiv 0$ and $\bar{\varphi}(\xi)>0$ for $\xi\geq 0$. Since $\underline{\psi}(\xi)\equiv 0$ and $\underline{\varphi}'(\xi)<0$ for $\xi\leq 0$, we see that $\underline{\varphi}$ and $\bar{\varphi}$ satisfy
 \bess\left\{\begin{array}{ll}
 s\underline{\varphi}'-\underline{\varphi}''
 \le\underline{\varphi}(1-\underline{\varphi}), \ \ & \xi<0,\\[1mm] s\bar{\varphi}'-\bar{\varphi}''> \bar{\varphi}(1-\bar{\varphi}-a\underline{\psi})
 =\bar{\varphi}(1-\bar{\varphi}), \ \ & \xi<0,\\[1mm]
 \underline{\varphi}(-\infty)=1=\bar{\varphi}(-\infty), \ \ \underline{\varphi}(0)=0<\bar{\varphi}(0).
 \end{array}\right.\eess
In view of Lemma \ref{p3.1}, we have $\underline{\varphi}(\xi)\leq \bar{\varphi}(\xi)$ for $\xi\leq 0$. Similarly, since $\bar\psi'(\xi)>0$ for $\xi\geq 0$, we have
\bess
 s\bar{\psi}'-d\bar{\psi}''\ge r\bar{\psi}(1-\bar{\psi}),
  \ \ \ \xi\geq 0.
 \eess
Clearly,
 \[s\underline{\psi}'-d\underline{\psi}''< r\underline{\psi}(1-\underline{\psi}-b\bar{\varphi})
 <r\underline{\psi}(1-\underline{\psi}),
  \ \  \ \xi\geq 0.\]
Thus the second inequality of (\ref{3.4}) is also satisfied.

{\bf Step 3}: Existence of a solution to (\ref{3.2}) when $s\in[0,s^*)$.

Define $\chi(\xi)=1$ for $\xi\geq 0$ and $\chi(\xi)=0$ for $\xi<0$.  Let  $(p(t,\xi),q(t,\xi))$ be the unique solution of
  \bes\left\{\begin{array}{ll}
 p_t-p_{\xi\xi}+s p_\xi
 =p(1-p-a\chi q),\ \ &t>0, \ \ \xi\in\mathbb{R},\\[1mm]
  q_t-d q_{\xi\xi}+s q_\xi=r q(1- q-bp),\ \ &t>0, \ \ \xi>0,\\[1mm]
  q(t,0)=0, &t> 0,\\[1mm]
 p(0,\xi)=\underline{\varphi}(\xi), \ \ &\xi\in\mathbb{R}, \\[1mm]
  q(0,\xi)=\bar{\psi}(\xi), \ \ &\xi\geq 0.
 \end{array}\right.\lbl{3.12}\ees
By the strong maximum principle, we have  $p(t,\xi)>0$ for $t>0$, $\xi\in\mathbb{R}$ and $ q(t,\xi)>0$ for $t>0$, $\xi>0$. 

Moreover, in view of the differential inequalities satisfied by $(\underline\varphi, \bar\psi)$, and the order-preserving property of competition systems,  we also have that $p(t,\xi)$ and $ q(t,\xi)$ are increasing and decreasing in $t$, respectively. 

Furthermore, due to the differential inequalities satisfied by $(\bar{\varphi}, \underline{\psi})$, and  (\ref{3.4}), by use of the comparison principle once again, we get that
 \[
p(t,\xi)\leq\bar{\varphi}(\xi) \ \ \mbox{for} \ t>0, \ \xi\in\mathbb{R};
 \quad \ \  q(t,\xi)\geq\underline{\psi}(\xi) \ \ \mbox{for} \ t,\,\xi>0.
\]
 Therefore, the limits
 \bess
 \lim_{t\to\infty} p(t,\xi)=\varphi_*(\xi), \ \ \ \lim_{t\to\infty} q(t,\xi)=\psi_*(\xi)
 \eess
exist, and satisfy 
\[\mbox{
$\underline\varphi(\xi)\leq \varphi_*(\xi)\le \bar\varphi(\xi)$ for $\xi\in\mathbb{R}$, $\bar\psi(\xi)\geq \psi_*(\xi)\ge\underline{\psi}(\xi)$ for $\xi>0$.}
\]
It follows that 
 \bes
 \varphi_*(-\infty)=1, \; \varphi_*(\infty)=0,\ \ \psi_*(0)=0, \ \ \psi_*(\infty)= 1.
 \lbl{3.14}\ees
Moreover, upon setting  $\psi_*(\xi)=0$ for $\xi<0$, it is easily seen that $(\varphi_*,\psi_*)$ satisfies the differential equations in (\ref{3.2}). By the 
Hopf boundary lemma,  we have
 \bes
 \psi_*'(0)>0.\lbl{3.15}\ees

It remains to show that $\varphi'_*<0$ $(\forall\xi\in\mathbb R)$ and $\psi_*'>0$
$(\forall \xi\geq 0)$.  Since $\underline\varphi'\leq 0$ and $\bar\psi'\geq 0$, we may use the maximum principle to the cooperative system satisfied by 
$(w(t,\xi), z(t,\xi)):=(-p_\xi(t,\xi), q_\xi(t,\xi))$ to conclude that 
\[
p_\xi(t,\xi)\leq 0 \; (\forall t>0, \forall \xi\in\mathbb R),\; q_\xi(t,\xi)\geq 0\; (\forall t>0, \forall \xi\geq 0).
\]
It follows that
\[
 \varphi'_*\leq 0 \; (\forall\xi\in\mathbb R),\;\;  \psi_*'\geq 0
\; (\forall \xi\geq 0).
\]
Applying the strong maximum principle to the cooperative system satisfied by $(-\varphi'_*, \psi_*')$,
 we further obtain
 \[
  \varphi_*'(\xi)<0 \  (\forall \ \xi\in\mathbb{R}),\;\; \psi_*'(\xi)>0 \ (\forall  \xi\geq 0).
  \]
Hence $(\varphi_*,\psi^*)$ is a solution of (\ref{3.2}).
\end{proof}

\begin{lem}\lbl{uniqueness}
For every $s\in [0, s^*)$, \eqref{3.2} has a unique solution.
\end{lem}
\begin{proof}
Let $(\varphi,\psi)$ be an arbitrary solution of (\ref{3.2}). We shall prove that
 \bes
 \varphi(\xi)=\varphi_*(\xi) \ (\forall \xi\in\mathbb{R}),
\quad \psi(\xi)=\psi_*(\xi)
 \ (\forall  \xi\geq 0),
  \lbl{3.19}\ees
where $(\varphi_*,\psi_*)$ is the solution of (\ref{3.2}) obtained in Step 3 of the proof of Lemma \ref{existence}.
As before we take $\psi(\xi)=\psi^*(\xi)=0$ for $\xi\le 0$.

 We are going to prove \eqref{3.19} in four steps, involving a ``sliding method'' (see Steps 3 and 4).

{\bf Step 1}. We show that
 \bes
  \varphi_*(\xi)\leq \varphi(\xi)\ (\forall \xi\in\mathbb{R}),\;
\quad \psi(\xi)\leq \psi_*(\xi)
 \ (\forall  \xi\geq  0).
 \lbl{3.21}\ees
 The argument leading to (\ref{3.4}) can be repeated here to yield
 \bess
 \underline{\varphi}(\xi)\leq\varphi(\xi) \ (\forall  \xi\in\mathbb{R}),\;
\quad \psi(\xi)\le\bar{\psi}(\xi)
 \  (\forall  \xi\geq 0),
  \eess
where $(\underline{\varphi},\bar {\psi})$ is given in Step 1 of the proof of Lemma \ref{existence}.
 So the solution $(p(t,\xi),q(t,\xi))$ of (\ref{3.12}) satisfies
  \bess
 p(t,\xi)\leq\varphi(\xi) \ (\forall  t\ge 0, \ \forall \xi\in\mathbb{R}),
\quad \psi(\xi)\le q(t,\xi) \ (\forall  t\ge 0, \ \forall \xi\geq 0),
  \eess
which clearly implies \eqref{3.21}.

{\bf Step 2}. Asymptotic expansions of $\vp(\xi)$ and $\psi(\xi)$ as $\xi\to+\ty$.

A simple calculation shows that the first order ODE system satisfied by $(\vp,\vp',\psi, \psi')$ has a critical point at $(0,0,1,0)$, which is a saddle point. Therefore by standard stable manifold theory  (see, e.g., Theorem 4.1 and its proof in Chapter 13 of \cite{CL}), 
\[
\vp(\xi)\to 0,\; 1-\psi(\xi)\to 0 \mbox{ exponentially as } \xi\to+\infty.
\]
The equations satisfied by $\vp$ and  $w:=1-\psi$ may be writen in the form 
\bes\lbl{phi-w}
\left\{\begin{array}{l}
\vp''-s\vp'-(a-1)\vp+\epsilon_1(\xi)\vp=0,\smallskip\\[1mm]
dw''-sw'-rw+rb\vp+\epsilon_2(\xi)w=0,
\end{array}
\right.
\ees
where
\[
\epsilon_1(\xi):=a w(\xi)-\vp(\xi)\to 0\mbox{ exponentially as } \xi\to+\infty,
\]
and
\[
\epsilon_2(\xi):=rw(\xi)-rb\vp(\xi)\to 0 \mbox{ exponentially as } \xi\to+\infty.
\]

Set
\[
\gamma_1:=\frac{s-\sqrt{s^2+4(a-1)}}{2},\;\; \gamma_2:=\frac{s+\sqrt{s^2+4(a-1)}}{2},
\]
and
\[g(y)=-dy^2+sy+r,\; \lambda_1=\frac{s-\sqrt{s^2+4rd}}{2d},\;\; \lambda_2=\frac{s+\sqrt{s^2+4rd}}{2d}.
\]
Then define
\[ 
u_1:=e^{\gamma_1 \xi},\; u_2:=e^{\gamma_2\xi},
\]
and
\[
v_1:=\left\{\begin{array}{ll}
\frac{rb}{g(\gamma_1)}e^{\gamma_1\xi}, & \gamma_1\not=\lambda_1,\medskip\\
\frac{rb}{g'(\gamma_1)}\xi e^{\gamma_1\xi}, & \gamma_1=\lambda_1,
\end{array}\right.  \;\; v_2:=\left\{\begin{array}{ll}
\frac{rb}{g(\gamma_2)}e^{\gamma_2\xi}, & \gamma_2\not=\lambda_2,\medskip\\
\frac{rb}{g'(\gamma_2)}\xi e^{\gamma_2\xi}, & \gamma_2=\lambda_2,
\end{array}\right.  \;\; v_3:=e^{\lambda_1\xi},\;\; v_4:=e^{\lambda_2\xi}.
\]
It is easily seen that
\[
{\bf u_1}:=(u_1, v_1),\; {\bf u_2}:=(u_2, v_2),\; {\bf u_3}:=(0, v_3),\; {\bf u_4}:=(0, v_4)
\]
are linearly independent solutions of the linear system
\[
\left\{
\begin{array}{l}
u''-su'-(a-1)u=0,\\[1mm]
dv''-sv'-rv+rbu=0.
\end{array}\right.
\]
We are now in a position to apply  to the system \eqref{phi-w} Theorem 8.1 in Chapter 3 of \cite{CL} (for the case $\gamma_1\not=\lambda_1$ and $\gamma_2\not=\lambda_2$), or a variant of this result (see Question 35 in Chapter 3 of \cite{CL} or Theorem 13.1 in Chapter X of \cite{H}, for the remaining cases), 
to conclude that \eqref{phi-w} has four linearly independent solutions ${\bf \tilde u_i}$, $i=1,2,3,4$, satisfying
\[
{\bf \tilde u_i}(\xi)=(1+o(1)){\bf u_i}(\xi) \mbox{ as } \xi\to+\infty,\; i=1,2,3,4.
\]
Since $(\vp,w)$ solves \eqref{phi-w}, there exist constants $c_i$, $i=1,2,3,4$, such that
\[
(\vp,w)=\sum_{i=1}^4 c_i{\bf \tilde u_i}.
\]
Since $\lambda_2>0,\; \gamma_2>0$ and $\vp(+\infty)=w(+\infty)=0$, we necessarily have $c_2=c_4=0$.
Using  $\vp(\xi)>0$ and  $w(\xi)>0$ we  deduce that $c_1>0$, and in the case $\gamma_1<\lambda_1$,
we further have $c_3>0$.  We thus obtain,
 as $\xi\to+\infty $,
 \bes
 \vp(\xi)&=&c_1{\rm e}^{\gamma_1\xi}(1+o(1)),\nonumber\\[1mm]
 w(\xi)&=&\left\{\begin{array}{ll}\medskip
 \dd c_1\frac{rb}{g(\gamma_1)}{\rm e}^{\gamma_1\xi}(1+o(1)) \ \ \ &\mbox{if} \ \ \gamma_1>\lambda_1,\\\medskip
 \dd c_1\frac{rb}{g'(\gamma_1)}\xi{\rm e}^{\gamma_1\xi}(1+o(1)) \ \ \ &\mbox{if} \ \ \gamma_1=\lambda_1,\\
 \dd c_3{\rm e}^{\lambda_1\xi}(1+o(1)) \ \ \ &\mbox{if} \ \ \gamma_1<\lambda_1.
 \end{array}\rr.\nonumber
 \ees
In other words, there exist positive constants  $C_\vp$ and $C_\psi$ such that, as $\xi\to+\infty$,
\bes\lbl{vp}
\vp(\xi)=C_\vp{\rm e}^{\gamma_1\xi}(1+o(1)),
\ees
 \bes
 \psi(\xi)&=&\left\{\begin{array}{ll}\medskip
 1-C_\psi{\rm e}^{\gamma_1\xi}(1+o(1)) \ \ \ &\mbox{if} \ \ \gamma_1>\lambda_1,\\\medskip
 1-C_\psi\xi{\rm e}^{\gamma_1\xi}(1+o(1)) \ \ \ &\mbox{if} \ \ \gamma_1=\lambda_1,\\
 1-C_\psi{\rm e}^{\lambda_1\xi}(1+o(1)) \ \ \ &\mbox{if} \ \ \gamma_1<\lambda_1.
 \end{array}\rr.
 \lbl{psi}\ees

{\bf Step 3}.  We show the existence of some constant $k_0>0$ such that, for all $k\ge k_0$,
 \bes\lbl{k0}
\varphi_{*}(\xi)\geq \varphi(\xi+k) \ (\forall\xi\in \mathbb{R}),\quad \psi(\xi+k)\geq \psi_*(\xi)\ (\forall\xi\geq 0).
\ees

Due to $\varphi(+\infty)=0$, there exists $k_1>0$ such that $\varphi_*(0)\geq \varphi(k)$ for all $k\ge k_1$. Denote $\varphi_{k}(\xi)=\varphi(\xi+k)$. Then $\varphi_{k}$ and $\varphi_*$ satisfy
 \bess
 &s\varphi'_{k}-\varphi''_{k}\leq \varphi_{k}(1-\varphi_{k}),\ \ s\varphi_*'-\varphi_*''=\varphi_*(1-\varphi_*),\ \ \xi<0,&\\[1mm]
 &\varphi_{k}(-\infty)=\varphi_*(-\infty)=1, \ \ \ \varphi_{k}(0)=\varphi(k)\leq \varphi_*(0).&
 \eess
By Lemma \ref{p3.1} we deduce
 \[
\varphi_*(\xi)\geq \varphi_k(\xi) \  (\forall  \xi\le 0).
\]

The expansion in Step 2 above also holds for $\psi_*(\xi)$, namely, as $\xi\to+\infty$,
 \bes\lbl{psi*}
 \psi_*(\xi)&=&\left\{\begin{array}{ll}\medskip
 1-C^*{\rm e}^{\gamma_1\xi}(1+o(1)) \ \ \ &\mbox{if} \ \ \gamma_1>\lambda_1,\\\medskip
 1-C^*\xi{\rm e}^{\gamma_1\xi}(1+o(1)) \ \ \ &\mbox{if} \ \ \gamma_1=\lambda_1,\\
 1-C^*{\rm e}^{\lambda_1\xi}(1+o(1)) \ \ \ &\mbox{if} \ \ \gamma_1<\lambda_1,
 \end{array}\rr.
  \ees
where $C^*>0$. In view of (\ref{psi}) and  (\ref{psi*}), we can find $\xi_0\gg 1$ and $k_2\ge k_1$ such that
 \[\psi(\xi+k_2)\geq \psi_*(\xi) \ \ (\forall  \xi\ge\xi_0).\]
Since $\psi(\xi)$ is increasing, it follows that
 \[\psi_k(\xi):=\psi(\xi+k)\geq \psi_*(\xi) \ \ (\forall  k\ge k_2, \forall \xi\ge\xi_0).\]
 Hence for   $k_3=k_2+\xi_0$, we have
 \[\psi(\xi+k)\geq \psi_*(\xi) \ \ (\forall  k\ge k_3, \forall \xi\geq 0).\]

By Step 2 we have, for $\xi\to+\infty$,
 \bes\lbl{vp-vp*}
\varphi_{*}(\xi)=C_*{\rm e}^{\gamma_1\xi}(1+o(1)), \ \ \
 \varphi(\xi)=C_{\vp}{\rm e}^{\gamma_1\xi}(1+o(1)),
\ees
with $C_*>0$ and $C_\vp>0$. Hence, there exist $\xi_1\gg 1$ and $k_0\ge k_3$ such that
 \[
\varphi_{*}(\xi)\geq \varphi(\xi+k)\ \ (\forall  k\ge k_0, \forall \xi\ge\xi_1).
\]
Therefore $\varphi_{*}$ and $\varphi_k$ satisfy
 \bess
& s\varphi'_{*}-\varphi''_{*}\;=&\varphi_{*}(1-\varphi_{*}-a\psi_*),\ \ \xi>0,\\[1mm] 
&s\varphi_k'-\varphi_k''\;=&\varphi_k(1-\varphi_k-a\psi_k)\leq \varphi_k(1-\varphi_k-a\psi_*),\ \ \xi>0,\\[1mm]
 &\varphi_{*}(0)\geq \varphi_k(0), &\ \ \ \dd\limsup_{\xi\to\infty}\frac{\varphi_k(\xi)}{\varphi_{*}(\xi)}\le 1.
 \eess
By Lemma \ref{p3.1} we deduce 
\[
\mbox{$\varphi_{*}(\xi)\geq \varphi_k(\xi)$ for all $\xi\ge 0$.}
\]
Hence \eqref{k0} holds for $k\geq k_0$.

{\bf Step 4}. Completion of the proof.

Define
 \[
\bar k=\inf\{k_0>0:\, \varphi_{*}(\xi)\geq \varphi(\xi+k) \ \ \mbox{in} \ \mathbb{R}, \ \ \psi(\xi+k)\geq \psi_*(\xi) \ \ \mbox{in} \ [0,\infty), \ \ \forall  k\geq k_0\}.
\]
Clearly $\bar k\geq 0$ and
 \[
\varphi_{*}(\xi)\geq\varphi_{\bar k}(\xi)=\varphi(\xi+\bar k) \ \ \ \mbox{in} \ \ \mathbb{R}, \quad \ 
\psi_{\bar k}(\xi)=\psi(\xi+\bar k)\geq \psi_*(\xi)\ \ \ \mbox{in} \ \ [0,\infty).
\]
If  $\bar k=0$, the above inequalities and \eqref{3.21}  imply $(\vp,\psi)\equiv (\vp_*,\psi_*)$, as we wanted.

Suppose that $\bar k>0$. We are going to derive  a contradiction.
We observe that $\psi_*(\xi)$ and $\psi_{\bar k}(\xi)$ satisfy
 \bess
 &s\psi_*'-d\psi_*''\:\; =& r\psi_*(1-\psi_*-b\varphi_*)\leq r\psi_*(1-\psi_*-b\varphi_{\bar k}), \ \ \xi>0,\\[1mm]
 &s{\psi_{\bar k}}'-d{\psi_{\bar k}}''\;=& r\psi_{\bar k}(1-\psi_{\bar k}-b\varphi_{\bar k}), \ \ \xi>0,\\[1mm]
 &\psi_*(0)=0<\psi_{\bar k}(0),&\;\;  \psi_*(\infty)=\psi_{\bar k}(\infty)=1.
\eess
By the strong maximum principle, we obtain 
 \bes\psi_{\bar k}(\xi)>\psi_*(\xi)\ \ \ \mbox{in} \ \ [0,\infty).
 \lbl{3.26}\ees
Similarly we can show that
 \bes
 \varphi_{\bar k}(\xi)<\varphi_*(\xi) \ \ \ \mbox{in} \ \ \mathbb{R}.
 \lbl{3.27}\ees

 Define
 \[\phi:=\varphi_{*}-\varphi_{\bar k}, \ \ \zeta:=\psi_{\bar k}-\psi_*.\]
Then
 \[\phi(\xi)>0 \ \ \mbox{in} \ \ \mathbb{R},
 \quad \zeta(\xi)>0 \ \ \mbox{in} \ \ [0,\infty),\]
and  $(\phi,\zeta)$ satisfies
 \bess
\left\{\begin{array}{ll}\medskip
 s\phi'-\phi''=a\varphi_{\bar k}\zeta-(\varphi_{\bar k}+\varphi_{*}+a\psi_*-1)\phi, \ \ \xi\in\mathbb{R},\\[1mm]
 s\zeta'-d\zeta''=rb\psi_*\phi-r(\psi_{\bar k}+\psi_*+b\varphi_{\bar k}-1)\zeta, \ \ \xi>0,
 \end{array}\right.
\eess
which can be rewritten in the form
\bes\lbl{phi-zeta}
\left\{\begin{array}{ll}\medskip
\phi''- s\phi'+(1-a)\phi +\tilde\epsilon_1(\xi)\phi+\tilde\epsilon_2(\xi)\zeta=0, \ \ \xi\in\mathbb{R},\\[1mm]
 d\zeta''-s\zeta'-r\zeta+ rb\phi+\tilde\epsilon_3(\xi)\phi+\tilde\epsilon_4(\xi)\zeta=0, \ \ \xi>0,
 \end{array}\right.
\ees
where
\bess
&\tilde\epsilon_1(\xi):&=a[1-\psi_*(\xi)]-\vp_{\bar k}(\xi)-\vp_*(\xi),\; \tilde\epsilon_2(\xi):=a\vp_{\bar k}(\xi),\\[1mm]
&\tilde\epsilon_3(\xi):&=rb[\psi_*(\xi)-1], \; \tilde\epsilon_4(\xi):=r[2-\psi_{\bar k}(\xi)-\psi_*(\xi)]+b\vp_{\bar k}(\xi).
\eess
By our expansions in Step 2, we have
\[
\tilde\epsilon_i(\xi)\to 0 \mbox{ exponentially as } \xi\to+\infty,\; i=1,2,3,4.
\]
Therefore for the same reasons as in Step 2, we have, as $\xi\to+\ty$,
 \bes\lbl{phi}
\phi(\xi)=C_1{\rm e}^{\gamma_1\xi}(1+o(1)),
\ees
and
  \bes\lbl{zeta}
 \zeta(\xi)&=&\left\{\begin{array}{ll}\medskip
 C_2{\rm e}^{\gamma_1\xi}(1+o(1)) \ \ \ &\mbox{if} \ \ \gamma_1>\lambda_1,\\\medskip
 C_2\xi{\rm e}^{\gamma_1\xi}(1+o(1)) \ \ \ &\mbox{if} \ \ \gamma_1=\lambda_1,\\
 C_2{\rm e}^{\lambda_1\xi}(1+o(1)) \ \ \ &\mbox{if} \ \ \gamma_1<\lambda_1,
 \end{array}\rr.
 \ees
where $C_1, C_2$ are positive constants.

From \eqref{vp-vp*} and \eqref{phi}, we obtain
\[
C_*=C_\vp e^{\gamma_1\bar k}+C_1.
\]
Similarly, it follows from \eqref{psi}, \eqref{psi*} and \eqref{zeta} that
\[
C^*=C_\psi e^{\sigma \bar k}+C_2,
\;\; \mbox{
where $\sigma=\max\{\gamma_1,\lambda_1\}$.}
\]
Therefore, there exists $\epsilon_0>0$ sufficiently small so that
\[
C_*>C_\vp e^{\gamma_1 (\bar k-\epsilon)},\; C^*>C_{\psi}e^{\sigma (\bar k-\epsilon)} \;\mbox{ for all } \epsilon\in (0,\epsilon_0].
\]
In view of \eqref{vp-vp*}, \eqref{psi} and \eqref{psi*}, these inequalities imply that, for all large $\xi$, say $\xi\geq M>0$, we have
\[
\vp_*(\xi)\geq \vp(\xi+\bar k-\epsilon),\; \psi_*(\xi)\leq \psi(\xi+\bar k-\epsilon) \; (\forall \epsilon\in (0,\epsilon_0]).
\]
Since
\[
\vp_*(\xi)>\vp(\xi+\bar k) \mbox{ and } \psi_*(\xi)<\psi(\xi+\bar k) \mbox{ for } \xi\in [0, M],
\]
by continuity, we can find $\epsilon_1\in (0,\epsilon_0]$ such that
\[
\vp_*(\xi)\geq \vp(\xi+\bar k-\epsilon_1),\; \psi_*(\xi)\leq \psi(\xi+\bar k-\epsilon_1) \; (\forall \xi\in [0, M]).
\]
We thus obtain 
\bes\lbl{R+}
\vp_*(\xi)\geq \vp(\xi+\bar k-\epsilon_1),\; \psi_*(\xi)\leq \psi(\xi+\bar k-\epsilon_1) \; (\forall \xi\geq 0).
\ees
For $\xi<0$, we can use Lemma \ref{p3.1} and $\vp_*(0)\geq \vp_{\bar k-\epsilon_1}(0)$ to deduce
\bes\lbl{R-}
\vp_*(\xi)\geq \vp_{\bar k-\epsilon_1}(\xi) \; (\forall \xi\leq 0).
\ees
In view of the monotonicity of $\vp$ and $\psi$, and the definition of $\bar k$,
we deduce from \eqref{R+} and \eqref{R-} that $\bar k\leq\bar k-\epsilon_1$. This contradiction 
shows that $\bar k>0$ cannot happen, and the proof is complete.
\end{proof}

\begin{lem}\lbl{s*=s0}
For $s\geq s_0$, problem \eqref{3.2} has no solution. In view of Lemma \ref{existence}, this in particular implies that $s^*=s_0$.
\end{lem}

\begin{proof}
 Suppose on the contrary  that for some $s\geq s_0$, \eqref{3.2} has a solution $(\hat\varphi,\hat\psi)$. 
By Proposition \ref{kan-on}, for such an  $s$, \eqref{3.1} also has a solution, which we denote by  $(\Phi,\Psi)$. 
We are going to use  the sliding method again to derive a contradiction.

We first observe that $\hat\vp(\xi)$ and $\Phi(\xi)$ can be expanded near $\xi=+\infty$ in the form \eqref{vp}, while $\hat\psi(\xi)$ and $\Psi(\xi)$ can be expanded near $\xi=+\infty$ in the form \eqref{psi}.
This, together with the fact that $\Phi'(\xi)<0$ and $\Psi'(\xi)>0$, implies the existence of some  $k_0>0$ such that
 \[
\Phi(\xi+k)\leq \hat \vp(\xi),\; \Psi(\xi+k)\geq\hat\psi(\xi), \ \ \ \forall\ \xi\geq 0, \ \ k\geq k_0.
\]
Clearly 
\[
\Psi(\xi+k)>0=\hat\psi(\xi) \mbox{ for } \xi<0.
\]
Now we prove that 
\[
\mbox{$\Phi(\xi+k)\leq \hat\varphi(\xi)$ for all $\xi\in\mathbb{R}$ and $k\geq k_0$.}
\]
 We only need to show this for $\xi<0$. Denote, for $k\geq k_0$,  
\[
\mbox{$\Phi_k(\xi):=\Phi(\xi+k)$ and $\Psi_k(\xi):=\Psi(\xi+k)$;}
\]
 then 
 \bes
 &s\Phi'_k-\Phi''_k= \Phi_k(1-\Phi_k-a\Psi_k)
 \le\Phi_k(1-\Phi_k-a\hat\psi),\ \ \xi\in\mathbb{R},&
 \lbl{3.37}\\[1mm]
&s\hat\varphi'-\hat\varphi''=\hat\varphi(1-\hat\varphi-a\hat\psi),\ \ \xi\in\mathbb R,&\nonumber\\
 &\Phi_k(-\infty)=1=\hat\varphi(-\infty), \ \  \Phi_k(0)\leq \hat\varphi(0).\nonumber&
\ees
Applying Lemma \ref{p3.1} we deduce
$\Phi_k(\xi)\leq\hat\varphi(\xi)$ in $(-\infty,0]$. 

We are now able to define
 \[k^*=\inf\{k_0\in\mathbb{R}:\, \Phi_k(\xi)\leq \hat\varphi(\xi) \ \ \mbox{in} \ \mathbb{R}, \ \ \Psi_k(\xi)\geq \hat\psi(\xi) \ \ \mbox{in} \ [0,\infty), \ \ \forall \ k\geq k_0\}.\]
We claim that $k^*=-\infty$. Otherwise, $k^*$ is finite and we have, by continuity,
 \[
\Phi_{k^*}(\xi)\leq\hat\varphi(\xi) \ \ \ \mbox{in} \ \ \mathbb{R}, \quad \ \Psi_{k^*}(\xi)\geq\hat\psi(\xi)\ \ \ \mbox{in} \ \ 
[0,\infty).
\]
We note that the inequality (\ref{3.37}) still holds for $k=k^*$,  and this inequality is strict for $\xi<0$ due to $\Psi(\xi+k^*)>0=\hat\psi(\xi)$ for such $\xi$. Hence $\Phi_{k^*}(\xi)\not\equiv\hat\varphi(\xi)$, and by the strong maximum principle we obatin 
\[
\mbox{$\Phi_{k^*}(\xi)<\hat\varphi(\xi)$ for $x\in\mathbb R$.}
\]

We now have
  \bess\left\{\begin{array}{ll}
  s\Psi'_{k^*}-d\Psi''_{k^*}=r\Psi_{k^*}(1-\Psi_{k^*}-b\Phi_{k^*}),\ \ &\xi>0,\\[1mm]
  s\hat\psi'-d\hat\psi''=r\hat\psi(1-\hat\psi-b\hat\varphi)<r\hat\psi(1-\hat\psi-b\Phi_{k^*}),\ \ &\xi>0,\\[1mm]
\Psi_{k^*}(0)>0=\hat\psi(0), \ \ \Psi_{k^*}(+\infty)=\hat\psi(+\infty)=1.&
 \end{array}\right.\eess
It follows from Lemma \ref{p3.1} and the strong maximum principle that 
\[
\mbox{$\Psi_{k^*}(\xi)>\hat\psi(\xi)$ in $[0,\infty)$. }
\]
We may now use the expansion of $(\hat\vp-\Phi_{k^*}, \Psi_{k^*}-\hat\psi)$ near $\xi=+\infty$ as in Step 4 of the proof of Lemma \ref{uniqueness} to derive that 
\[
\Phi_{k^*-\epsilon}(\xi)\leq \hat\varphi(\xi), \; \Psi_{k^*-\epsilon}(\xi)\geq\hat\psi(\xi) \mbox{ for all $ \xi\in [0,+\infty)$
and  some small $\epsilon>0$. }
\]
Since $\Psi_{k^*-\epsilon}(\xi)>0=\hat\psi(\xi)$ for $\xi<0$, we see that \eqref{3.37} holds for $k=k^*-\epsilon$, and we can thus use Lemma \ref{p3.1} to deduce
\[
\Phi_{k^*-\epsilon}(\xi)\leq \hat\varphi(\xi)\; \mbox{ for } \xi<0.
\]
Due to the monotonicity of $\Phi$ and $\Psi$, we can now conclude that, for all $k\geq k^*-\epsilon$,
\[
\Phi_{k}(\xi)\leq \hat\varphi(\xi)\; (\forall \xi\in\mathbb R), \; \Psi_{k}(\xi)\geq\hat\psi(\xi) \; (\forall \xi\geq 0).
\]
It follow that $k^*\leq k^*-\epsilon$. This contradiction proves our claim that $k^*=-\infty$.

The fact $k^*=-\infty$ implies that $\Psi(\xi+k)\geq \hat\psi(\xi)$ in $[0,\infty)$ for all $k\in\mathbb{R}$. Letting $k\to-\infty$ and using $\Psi(\xi+k)\to 0$ as $k\to-\infty$ we conclude that $\hat\psi(\xi)\leq 0$. This is a contradiction to the fact that  $(\hat\varphi,\hat\psi)$ is a solution of (\ref{3.2}). This completes the proof of the lemma.
\end{proof}

\begin{lem}\lbl{lem-s1-s2} Let $(\vp_s, \psi_s)$ denote the unique solution of \eqref{3.2} with $s\in [0, s_0)$. Then
 $0\leq s_1<s_2<s_0$ implies
\bes\lbl{s1-s2}
\psi_{s_1}'(0)>\psi_{s_2}'(0),  \; \psi_{s_1}(\xi)>\psi_{s_2}(\xi) \; (\forall \xi>0), \; \vp_{s_1}(\xi)<\vp_{s_2}(\xi) \; (\forall \xi\in\R).
\ees
\end{lem}
\begin{proof}
To simplify the notations we denote $(\varphi_i,\psi_i)=(\varphi_{s_i},\psi_{s_i})$, $i=1,2$. Similarly to the proof of (\ref{3.4}) we can show that
 \bess
 \underline{\varphi}(\xi)\leq\varphi_2(\xi), \ \ \forall \ \xi\in\mathbb{R};
\quad \psi_2(\xi)\le\bar{\psi}(\xi),
 \ \ \forall \ \xi\geq 0,
  \eess
where $(\underline{\varphi},\bar {\psi})$ is given in Step 1 of the proof of Lemma \ref{existence}. It is obvious that $(\varphi_2,\psi_2)$ satisfies
 \bess\left\{\begin{array}{ll}
 s_1\varphi'_2-\varphi''_2>\varphi_2(1-\varphi_2-a\psi_2),\ \ \ &\xi\in\mathbb{R},\\[1mm]
 s_1\psi'_2-d\psi''_2<r\psi_2(1-\psi_2-b\varphi_2),\ \ &\xi>0.
 \end{array}\right.\eess
By the comparison principle we derive that the solution $(p(t,\xi),q(t,\xi))$ of (\ref{3.12}) with $s=s_1$ satisfies
  \bess
 p(t,\xi)\leq\varphi_2(\xi), \ \ \forall \ t\ge 0, \ \ \xi\in\mathbb{R};
\quad q(t,\xi)\ge\psi_2(\xi),
 \ \ \forall \ t\ge 0, \ \ \xi\geq 0,
  \eess
which implies
  \[\varphi_1(\xi)\leq\varphi_2(\xi), \ \ \forall \ \xi\in\mathbb{R};
\quad \psi_1(\xi)\ge\psi_2(\xi),
 \ \ \forall \ \xi\geq 0,\]
since $p(t,\xi)\to\vp_1(\xi)$ and $q(t,\xi)\to \psi_1(\xi)$ as $t\to\infty$ (cf. Step 3 in the proof of Lemma \ref{existence} and the conclusion of Lemma \ref{uniqueness}).
Furthermore, the strong maximum principle yields
 \[\varphi_1(\xi)<\varphi_2(\xi), \ \ \forall \ \xi\in\mathbb{R};
\quad \psi_1(\xi)>\psi_2(\xi),
 \ \ \forall \ \xi>0.\]
We now consider $w:=\psi_1-\psi_2$, which satisfies
\[
s_2w'-dw''> rw(1-\psi_1-\psi_2-b\vp_2),\; w>0 \; (\forall \xi>0);\; w(0)=0.
\]
By the Hopf boundary lemma, we deduce $w'(0)>0$, that is, $\psi_1'(0)>\psi'_2(0)$.
\end{proof}

\begin{lem}\lbl{map}
The operator $s\mapsto (\vp_s, \psi_s)$ is continuous from $[0,s_0)$ to $C_{loc}^2(\mathbb R)\times C^2_{loc}([0,+\infty))$.
Moreover, 
\bes\lbl{s-s0}
\mbox{$\lim_{s\to s_0}(\vp_s,\psi_s)=(1, 0)$ in $C_{loc}^2(\mathbb R)\times C^2_{loc}([0,+\infty))$.}
\ees
\end{lem}
\begin{proof}
Let $\{s_n\}$ be an arbitrary sequence contained in $[0, s_0)$ satisfying $\lim_{n\to\infty} s_n=\hat s\in [0, s_0]$. We first consider the case $\hat s<s_0$. We want to show that $\{(\vp_{s_n}, \psi_{s_n})\}$ has a subsequence that converges to
$(\vp_{\hat s}, \psi_{\hat s})$ in $C_{loc}^2(\mathbb R)\times C^2_{loc}([0,+\infty))$. The required continuity of the map $s\mapsto (\vp_s, \psi_s)$ is clearly a consequence of this conclusion.

Fix $s^1\in (\hat s, s_0)$.
By passing to a subsequence we may assume that $0\leq s_n\leq s^1$ for all $n\geq 1$. We thus obtain, by Lemma \ref{lem-s1-s2},
\[
\vp_0\leq \vp_{s_n}\leq \vp_{s^1} \; (\forall \xi\in \mathbb R),\; \psi_0\geq \psi_{s_n}\geq \psi_{s^1} \ (\forall \xi\geq 0).
\]
By standard $L^p$ regularity and the Sobolev embedding theorem,  we may assume that, subject to passing to a subsequence,
\[
(\vp_{s_n}, \psi_{s_n})\to (\hat\vp,\hat\psi) \mbox{ in $C_{loc}^1(\mathbb R)\times C^1_{loc}([0,+\infty))$ as $n\to\infty$.} 
\]
The Schauder theory then infers that the above convergence also holds in $C_{loc}^2(\mathbb R)\times C^2_{loc}([0,+\infty))$. Moreover, $(\hat\vp,\hat\psi)$ solves \eqref{3.2} with $s=\hat s$, except that we only have
$\hat\vp'\leq 0$ and $\hat\psi'\geq 0$. We note that the required asymptotic behavior of $(\hat\vp, \hat\psi)$ at $\xi=\pm\infty$ follows from
\[
\vp_0\leq \hat\vp\leq \vp_{s^1},\; \psi_0\geq \hat\psi\geq \psi_{s^1}.
\]
Applying the strong maximum principle to the system satisfied by $(-\hat\vp', \hat\psi')$, we deduce $\hat\vp'<0$
in $\mathbb R$ and $\hat\psi'>0$ in $[0, +\infty)$. Thus $(\hat\vp, \hat\psi)$ is a solution of \eqref{3.2} with $s=\hat s$. By uniqueness, we necessarily have $(\hat\vp, \hat\psi)=(\vp_{\hat s}, \psi_{\hat s})$. Clearly this implies the required continuity of the map $s\mapsto (\vp_s, \psi_s)$.

We next consider the case $\hat s=s_0$ and prove \eqref{s-s0}. In this case we have
\[
\vp_0\leq \vp_{s_n}\leq 1 \; (\forall \xi\in \mathbb R),\; \psi_0\geq \psi_{s_n}\geq 0 \ (\forall \xi\geq 0).
\]
Hence we may repeat the above argument to conclude that, subject to passing to a subsequence,
\[
(\vp_{s_n}, \psi_{s_n})\to (\vp^0,\psi^0) \mbox{ in $C_{loc}^2(\mathbb R)\times C^2_{loc}([0,+\infty))$ as $n\to\infty$.} 
\]
and $(\vp^0,\psi^0)$ solves \eqref{3.2} with $s=s_0$, except that we only have
\[
(\vp^0)'\leq 0,\; 
\vp_0\leq \vp^0\leq 1 \; (\forall \xi\in \mathbb R);\;  (\psi^0)'\geq 0,\; \psi_0\geq\psi^0\geq 0 \ (\forall \xi\geq 0).
\]

If $\psi^0\equiv 0$, then $\vp^0$ satisfies $0<\vp_0\leq \vp^0\leq 1$ and
\[
-(\vp^0)''\geq \vp^0(1-\vp^0) \mbox{ for }\xi\in \mathbb R.
\]
For large $R>1$, let $u_R$ be the unique positive solution of 
\[
-u''=u(1-u) \mbox{ in } (-R, R),\; u(-R)=u(R)=0.
\]
It is well known that (see \cite{DMa}) $u_R\to 1$ in $C^2_{loc}(\mathbb R)$ as $R\to+\infty$. By Lemma 2.1 of
\cite{DMa}, we have $1\geq \vp^0\geq u_R$ in $[-R, R]$. Letting $R\to\infty$ we obtain $\vp^0\equiv 1$, as we wanted.

Next we show that $\psi^0\not\equiv 0$ leads to a contradiction. In such a case, by monotonicity, we have $\psi^0(+\infty)\in (0, 1]$.
If $\vp^0(+\infty)=0$,  then from the differential equation for $\psi^0$ we deduce $\psi^0(+\infty)=1$, and so $(\vp^0,\psi^0)(+\infty)=(0, 1)$.
Let $(\Phi_0, \Psi_0)$ be a solution of \eqref{3.1} with $s=s_0$. Then we can repeat the 
sliding method in the proof of Lemma \ref{s*=s0} to deduce that
\[
\psi^0(\xi)\leq \Psi_0(\xi+k) \mbox{ for } \xi\geq 0,\; k\in\mathbb R.
\]
Letting $k\to-\infty$ we obtain $\psi^0\leq 0$, which is a contradiction.  

If $\vp^0(+\infty)\in (0, 1]$, then from the differential equation for $\psi^0$ necessarily 
\[
\psi^0(+\infty)=1-b\vp^0(+\infty)\in (0, 1).
\]
Hence
\bes\lbl{s0-infty}
\vp^0(+\infty)>\Phi_0(+\infty),\;\; \psi^0(+\infty)<\Psi_0(+\infty).
\ees
In such a case we may use the sliding method in the proof of Lemma \ref{s*=s0} again (with obvious simplifications because the comparison of $(\vp^0,\psi^0)$ with 
$(\Phi_0(\cdot+k),\Psi_0(\cdot+k))$ near $\xi=+\infty$ can be carried out by \eqref{s0-infty} now), to deduce $\psi^0\leq 0$. So we also arrive at a contradiction. 
Thus only $(\vp^0,\psi^0)\equiv (1,0)$ is possible, which implies \eqref{s-s0}. The proof is complete.
\end{proof}

\begin{lem}\lbl{lem-s_mu}
For any $\mu>0$, there exists a unique $s=s_\mu\in (0, s_0)$ such that
\bes\lbl{s_mu}
\mu \psi'_{s_\mu}(0)=s_\mu. 
\ees
Moreover,
\bes\lbl{mu-infty}
\mu\mapsto s_\mu \mbox{ is strictly increasing and } \lim_{\mu\to+\infty}s_\mu=s_0.
\ees
\end{lem}
\begin{proof}
By Lemmas \ref{lem-s1-s2} and \ref{map},  for fixed $\mu>0$, the function $\eta_\mu(s):= \mu\psi_s'(0)-s$ is continuous and strictly decreasing
for $s\in [0, s_0)$. By \eqref{s-s0}, $\eta_\mu(s_0-0)=-s_0<0$. Clearly $\eta_\mu(0)=\mu\psi'_0(0)>0$. Therefore there exists a
unique $s=s_\mu\in (0, s_0)$ such that $\eta_\mu(s_\mu)=0$. This proves the first part of the lemma.

For fixed $s\in [0, s_0)$, $\eta_\mu(s)$ is strictly increasing in $\mu$. It follows that $s_\mu$ is strictly increasing in $\mu$.
Finally, for any $\epsilon>0$ and $s\in [0, s_0-\epsilon]$, we have $\eta_\mu(s)\geq \eta_\mu(s_0-\epsilon)\to+\infty$ as $\mu\to+\infty$.
It follows that $s_0-\epsilon<s_\mu<s_0$ for all large $\mu$. Clearly this implies $\lim_{\mu\to+\infty}s_\mu=s_0$.
\end{proof}

\begin{rem}\label{remark1}{\rm
Using the change of variables described in Section 1, which reduces \eqref{f1} to \eqref{p1}, we can
apply Theorem \ref{th3.1} to obtain a parallell result for the general case
\bes
\left\{
\begin{array}{ll}
 s\varphi'-d_2\varphi''=\varphi(a_2-b_2\psi -c_2\varphi)\ \ \ \ \ \ (\forall \xi\in\mathbb{R}),\\[1mm]
 s\psi'-d_1\psi''=\psi(a_1-b_1\psi-c_1\varphi)\ \ \ \ \ \ (\forall \xi>0),\\[1mm]
\varphi(-\infty)=\frac{a_2}{c_2},\; \varphi'<0\  (\forall\xi\in\mathbb R),\;  \vp(+\infty)=0,\\[1mm]
\psi\equiv 0 \ (\forall \xi\leq 0),\  \psi'>0 \  (\forall \xi\geq 0), \ \psi(+\infty)=\frac{a_1}{b_1},
 \end{array}\right.
\lbl{g-semi-wave}
\ees
when \eqref{sup} holds. Using the uniqueness of the semi-wave, it is easy to show that the unique value $s=s_\mu$ determined by
\[
\mu\psi'_s(0)=s
\]
 depends continuously on the parameters $\mu$ and  $a_i, b_i, c_i, d_i$, $i=1,2$. This observation will be used in the next section.
}
\end{rem}

\section{Spreading speed}
\setcounter{equation}{0}

In this section, we prove Theorem \ref{main}. We always assume that the conditions of Theorem \ref{main} holds.
We first prove

\begin{lem}\label{lem3.1}
\begin{equation}\label{u-bd}
\limsup_{t\to+\infty}\frac{h(t)}{t}\leq s_\mu.
\end{equation}
\end{lem}
\begin{proof}
For clarity we divide the proof into two steps.

{\bf Step 1}. We show that for any given small $\delta>0$ and large $T_0>0$, there exist positive constants $T$ and $M$ such that
\[ T>T_0, \;\; v(T,r)\geq 1-\delta \;\; (\forall r\geq M).
\]
Choose a constant $\sigma_0$ satisfying 
\[
0<\sigma_0<\liminf_{r\to+\infty} v_0(r), 
\]
and then consider the auxiliary problem
\begin{equation}
\label{w}
\left\{\begin{array}{ll}
w_t-w_{rr}=w(1-w), & t>0,  r>0,\\[1mm]
w=0, & t>0, r=0,\\[1mm]
w=\sigma_0, & t=0, r>0.
\end{array}
\right.
\end{equation}
It is easily seen that the unique solution of \eqref{w} has the property that
\[
\lim_{t\to+\infty} w(t,r)=w_*(r) \mbox{ locally uniformly in } r\in [0, +\infty),
\]
where $w_*(r)$ is the unique positive solution of 
\[
-w_*''=w_*(1-w_*) \;\; (r>0), \;\; w_*(0)=0.
\]
Moreover, $w_*'(r)>0$ and $w_*(+\infty)=1$. Therefore we can find $M_1>0$ and $T>T_0$ such that
\[
w(t, M_1)\geq w_*(M_1)-\delta/2\geq 1-\delta \;\; (\forall t\geq T).
\]

Applying the maximum principle to the equation satisfied by $\partial_r w(t,r)$, we deduce
$\partial_r w(t,r)\geq 0$ for $t>0$ and $r>0$. It follows that
\[
w(t, r)\geq 1-\delta \;\; (\forall t\geq T,\;\forall r\geq M_1).
\]

By the choice of $\sigma_0$, there exists $M_2>0$ such that $v_0(r)>\sigma_0$ for $ r\geq M_2$.
Set 
\[
M_3:=\max\{M_2, h(T)\}.
\]
Then for $t\in [0, T]$ we have $h(t)\leq M_3$, and hence $v$ satisfies
\[
v_t-\Delta v=v(1-v) \;\; (0<t\leq T,\; r>M_3),\;\; v(0, r)>\sigma_0 \;\; (r>M_3).
\]

Define
\[
\tilde w(t,r):= w(t, r-M_3).
\]
Then we have
\[
\tilde w_t-\tilde w_{rr}-\frac{N-1}{r}\tilde w_r\leq \tilde w_t-\tilde w_{rr}=\tilde w(1-\tilde w) \;\; (0<t\leq T,\; r>M_3).
\]
Since 
\[
\tilde w(t, M_3)=0<v(t, M_3) \;\; (\forall t>0),\; \tilde w(0, r)=\sigma_0<v(0, r)\;\; (\forall r>M_3),
\]
we can use the comparison principle to deduce
\[
v(t,r)\geq \tilde w(t, r) =w(t, r-M_3) \;\; (\forall 0<t\leq T,\;\forall r>M_3).
\]
Thus we may set 
\[ M:=M_1+M_3,
\]
and obtain
\[
v(T, r)\geq w(T, r-M_3)\geq w(T, M_1)\geq 1-\delta \; (\forall r\geq M).
\]
This completes the proof of Step 1.
\smallskip

{\bf Step 2}.  Proof of \eqref{u-bd} by the construction of a comparison triple $(\overline u(t,r),\underline v(t,r), \overline h(t))$.

By the comparison principle, we easily see that
\[
u(t,r)\leq U(t) \;\; (\forall t>0,\;\forall r\in [0, h(t)]),
\]
where $U(t)$ is the unique solution of
\[
U'=rU(1-U) \; (t>0),\;\; U(0)=\|u_0\|_\infty.
\]
Since $U(+\infty)=1$, for any given small $\delta>0$ 
we can find $T_0>0$ such that
\[
U(t)\leq 1+\delta \;\; (\forall t\geq T_0).
\]
It follows that
\[
u(t, r)\leq 1+\delta\; \; (\forall t\geq T_0,\;\forall r\in [0, h(t)]).
\]
By Step 1, we can find $T>T_0$ and $M>0$ such that
\[
v(T, r)\geq 1-\delta \;\; (\forall r\geq M).
\]
We thus obtain
\begin{equation}\label{t=T}
u(T, r)\leq 1+\delta \; \; (\forall r\in [0, h(T)],\;\; v(T, r)\geq 1-\delta  \;\; (\forall r\geq M).
\end{equation}

We now consider the auxiliary problem
 \bes\left\{\begin{array}{ll}
 s\varphi'-\varphi''=\varphi(1-\delta-\varphi-a\psi)\ \ \ \ \ \ (\forall \xi\in\mathbb{R}),\\[1mm]
 s\psi'-d\psi''=r\psi(1+2\delta-\psi-b\varphi)\ \ \ (\forall \xi>0),\\[1mm]
\varphi(-\infty)=1-\delta,\; \varphi'<0\  (\forall\xi\in\mathbb R),\;  \vp(+\infty)=0,\\[1mm]
\psi\equiv 0 \ (\forall \xi\leq 0),\  \psi'>0 \  (\forall \xi\geq 0), \ \psi(+\infty)=1+2\delta,
 \end{array}\right.\lbl{delta-sw}
\ees
Since $a>1>b$ and $\delta>0$ is small, by Remark \ref{remark1}, there exists a unique $s=s_\mu^\delta>0$ such that
for this value of $s$, \eqref{delta-sw} has a unique solution $(\vp^\delta,\psi^\delta)$, and
\[
\mu (\psi^\delta)'(0)=s_\mu^\delta,\;\;\; \lim_{\delta\to 0}s_\mu^\delta=s_\mu.
\]
Using $\varphi^\delta(+\infty)=0$ and $\psi^\delta(+\infty)=1+2\delta$, we can find $R>h(T)$ large so that
\bes\lbl{R}
\psi^\delta(R-h(T))>1+\delta,\;\; \varphi^\delta(R-M)<\inf_{r\in [0, M]} v(T, r).
\ees
We now define
\bess
\overline h(t):&=&s_\mu^\delta t+R,\;  \\
\overline u(t,r):&=&\psi^\delta(\overline h(t)-r),\\
 \underline v(t,r):&=&\varphi^\delta(\overline h(t)-r).
\eess
Clearly $\overline h(0)>h(T)$ and
\[
\overline h'(t)=s_\mu^\delta=\mu (\psi^\delta)'(0)=-\mu \overline u_r(t, \overline h(t))\;\; (\forall t>0).
\]
 Moreover,  by \eqref{t=T} and \eqref{R}
we have
\bess
\overline u(0, r)&=&\psi^\delta(R-r)\geq \psi^\delta (R-h(T)) >1+\delta\geq u(T, r) \;\; (\forall r\in [0, h(T)]),\\
\underline v(0,r)&=&\varphi^\delta(R-r)\leq \varphi^\delta(R-M)<v(T, r)\;\; (\forall r\in [0, M]).
\eess
For $r>M$, by \eqref{t=T},
\[
\underline v(0,r)=\vp^\delta(R-r)<\vp^\delta(-\infty)=1-\delta\leq v(T, r).
\]
Thus
\[
\underline v(0, r)<v(T, r) \;\; (\forall r\geq 0).
\]
Furthermore,
\bess
\overline u_t-d\Delta \overline u&=&s_\mu^\delta (\psi^\delta)'-d(\psi^\delta)''+d\frac{N-1}{r}(\psi^\delta)'\\
&\geq& r\psi^\delta(1+2\delta-\psi^\delta-b\vp^\delta)\\
&\geq & r\overline u(1-\overline u-b\underline v).
\eess
\bess
\underline v_t-\Delta \underline v&=& s_\mu^\delta(\vp^\delta)'-(\vp^\delta)''+\frac{N-1}{r}(\vp^\delta)'\\
&\leq &\vp^\delta(1-\delta-\vp^\delta-a\psi^\delta)\\
&\leq& \underline v(1-\underline v-a\overline u).
\eess
Finally we note that
\[
\overline u_r(t,0)=-(\psi^\delta)'(\overline h(t))<0,\;\; \underline v_r(t,0)=-(\vp^\delta)'(\overline h(t))>0.
\]
Hence we can apply the comparison principle in \cite{DLin1} (see Lemma 2.6 and Remark 2.7 there), to conclude that
\[
h(t+T)\leq \overline h(t) \;\; (\forall t>0).
\]
It follows that
\[
\limsup_{t\to+\infty}\frac{h(t)}{t}\leq s_\mu^\delta.
\]
Letting $\delta\to 0$ we obtain
\[
\limsup_{t\to+\infty}\frac{h(t)}{t}\leq s_\mu.
\]
\end{proof}

\begin{lem}\label{lem3.2}
\begin{equation}\label{l-bd}
\liminf_{t\to+\infty}\frac{h(t)}{t}\geq s_\mu.
\end{equation}
\end{lem}

To prove Lemma \ref{lem3.2}, we will use solutions of the following problem
\bes\left\{\begin{array}{ll}
 s\varphi'-\varphi''=\varphi(1+\delta-\varphi-a\psi)\ \ \ \ \ \ (\forall \xi\in\mathbb{R}),\\[1mm]
 s\psi'-d\psi''=r\psi(1-\delta-\psi-b\varphi)\ \ \ (\forall \xi>0),\\[1mm]
\varphi(-\infty)=1+\delta,\; \varphi'<0\  (\forall\xi\in\mathbb R),\;  \vp(+\infty)=0,\\[1mm]
\psi\equiv 0 \ (\forall \xi\leq 0),\  \psi'>0 \  (\forall \xi\geq 0), \ \psi(+\infty)=1-\delta,
 \end{array}\right.\lbl{1+delta-sw}
\ees
where $\delta>0$ is small. By Remark \ref{remark1}, there exists a unique $s=s_{\mu,\delta}>0$ such that
for this value of $s$, \eqref{1+delta-sw} has a unique solution $(\vp_\delta,\psi_\delta)$, and
\[
\mu \psi_\delta'(0)=s_{\mu,\delta},\;\;\; \lim_{\delta\to 0}s_{\mu,\delta}=s_\mu.
\]
 However, unlike in the proof of Lemma \ref{lem3.1}, here we need to modify $\vp_\delta$ first before we can use it
and $\psi_\delta$ to construct suitable comparison functions to prove \eqref{l-bd}.

\begin{lem}
\lbl{lem3.3}
For every large $\xi_0>0$, there exist a constant $\xi_1=\xi_1(\xi_0)>\xi_0$ and a function $\tilde\vp_\delta=\tilde\vp_{\delta,\xi_0}\in C^1(\R)$ such that
\bess
&&\limsup_{\xi_0\to+\infty}\left[\xi_1(\xi_0)-\xi_0\right]<+\infty,\; \tilde\vp_\delta'\leq 0,\; \tilde\vp_\delta\geq\vp_\delta \; (\forall \xi\in\R),\\
&&\tilde\vp_\delta\equiv \vp_\delta \; (\forall \xi\leq \xi_0),\;\; \tilde\vp_\delta\equiv\tilde\vp_\delta(\xi_1)>0\; (\forall \xi\geq \xi_1),\\
&&s\tilde\vp_\delta-\tilde\vp_\delta''\geq \tilde\vp_\delta(1-\tilde\vp_\delta-a\psi_\delta) \; (\forall\xi\in\R, \mbox{ in the weak sense}),\\
&&s\psi_\delta-d\psi_\delta''\leq r\psi_\delta(1-\psi_\delta-b\tilde\vp_\delta) \; (\forall\xi>0).
\eess
\end{lem}
\begin{proof}
Denote $\tilde a:=a(1-\delta)-\delta$ and 
\[
\tilde\gamma_1:=\frac{s_{\mu,\delta}-\sqrt{s_{\mu,\delta}^2+4(\tilde a-1)}}{2}.
\]
Then by standard ODE theory as used in the proofs of Lemmas \ref{(3.1)s_0} and \ref{uniqueness}, we have, similar to \eqref{vp},  
\bes\lbl{vp_delta}
\vp_\delta(\xi)=C_0e^{\tilde\gamma_1\xi}(1+o(1)),\; \vp_\delta'(\xi)=C_0\tilde\gamma_1e^{\tilde\gamma_1\xi}(1+o(1)) \mbox{ as } \xi\to+\infty,
\ees
 for some $C_0>0$.

Fix $\gamma>0$, with its value to be determined later,
and define, for $\xi\geq \xi_0$,
\[
\hat\vp(\xi):=\vp_\delta(\xi)+\vp_\delta(\xi_0)\left[e^{\gamma(\xi-\xi_0)}-1-\gamma(\xi-\xi_0)\right].
\]
By \eqref{vp_delta} we have, for large $\xi_0$ and $\xi\geq \xi_0$,
\bess
\hat\vp'(\xi)&=&\vp_\delta'(\xi)+\vp_\delta(\xi_0)\left[\gamma e^{\gamma(\xi-\xi_0)}-\gamma\right]\\
&=& C_0\tilde\gamma_1e^{\tilde\gamma_1\xi}(1+o(1))+C_0e^{\tilde\gamma_1\xi_0}(1+o(1))\left[\gamma e^{\gamma(\xi-\xi_0)}-\gamma\right]\\
&=&C_0 e^{\tilde\gamma_1\xi_0}\left[\tilde\gamma_1e^{\tilde\gamma_1z}+\gamma (e^{\gamma z}-1)\right](1+o(1)),
\eess
where $z:=\xi-\xi_0$. Since 
\[
\tilde\gamma_1<0<\gamma,
\]
the strictly increasing function
\[
f(z):=\tilde\gamma_1 e^{\tilde\gamma_1z}+\gamma (e^{\gamma z}-1)
\]
has a unique positive zero $z_0$. It follows that, for all large $\xi_0$,
\[
\hat\vp'(\xi_0+z_0+1)>0.
\]
Due to $\hat\vp'(\xi_0)=\vp_\delta'(\xi_0)<0$, there exists a unique $\xi_1=\xi_1(\xi_0)\in (\xi_0, \xi_0+z_0+1)$ such that
\[
\hat\vp'(\xi)<0 \; (\forall \xi\in[\xi_0,\xi_1)),\; \hat\vp'(\xi_1)=0.
\]
Define
\[
\tilde\vp_\delta(\xi)=\tilde\vp_{\delta,\xi_0}(\xi):=\left\{\begin{array}{ll}
\vp_\delta(\xi),& \xi\leq\xi_0,\\
\hat\vp(\xi), & \xi_0\leq\xi\leq\xi_1,\\
\hat\vp(\xi_1), & \xi\geq \xi_1.
\end{array}
\right.
\]
Then clearly $\tilde\vp_\delta\in C^1(\R)$,\; $\tilde\vp_\delta(\xi_1)>\vp_\delta(\xi_1)>0$,  and
\[
\limsup_{\xi_0\to+\infty}\left[\xi_1(\xi_0)-\xi_0\right]\leq z_0+1,\;\;\tilde\vp_\delta'\leq 0,\;\tilde\vp_\delta\geq\vp_\delta \; (\forall\xi\in\R).
\]
Moreover, it is also easily seen that
\[
\lim_{\xi_0\to+\infty} \|\tilde\vp_{\delta,\xi_0}-\vp_\delta\|_{L^\infty(\R)}=0.
\]
Thus for all large $\xi_0$,
\bess
s\psi_\delta'-d\psi_\delta''&=&r\psi_\delta(1-\delta-\psi_\delta-b\varphi_\delta)\\
&\leq& r\psi_\delta(1-\psi_\delta-b\tilde\vp_\delta)
\; \;\; (\forall \xi\in\R).
\eess

To complete the proof, it remains to show that, for every fixed large $\xi_0$, in the weak sense,
\bes\lbl{tilde-vp}
s\tilde\vp_\delta'-\tilde\vp_\delta''\geq \tilde\vp_\delta(1-\tilde\vp_\delta-a\psi_\delta) \; (\forall\xi\in\R).
\ees
Since $\tilde\vp_\delta$ is $C^1$, it suffices to show the above inequality  for
$\xi<\xi_0$, $\xi\in (\xi_0,\xi_1)$ and $\xi>\xi_1$ separately.

For $\xi<\xi_0$, \eqref{tilde-vp} follows directly from \eqref{1+delta-sw}. Since $a\psi_\delta(+\infty)=a(1-\delta)>1$,
we have
\[
1-\tilde\vp_\delta-a\psi_\delta<0\;\; (\forall \xi>\xi_1)
\]
provided that $\xi_0$ is large. Hence for every fixed large $\xi_0$, \eqref{tilde-vp} holds for $\xi>\xi_1$.

We next consider the case $\xi\in (\xi_0,\xi_1)$. Denote
\[
\eta(\xi):=e^{\gamma(\xi-\xi_0)}-1-\gamma(\xi-\xi_0).
\]
Then
\[
s\eta'- \eta''>-\gamma^2 e^{\gamma(\xi-\xi_0)} \;\; (\forall\xi>\xi_0),
\]
and hence, for $ \xi\in (\xi_0,\xi_1)$,
\bess
s\hat\vp'-\hat\vp''&=&s\vp_\delta'-\vp_\delta''+\vp_\delta(\xi_0)(s\eta'-\eta'')\\
&=&\vp_\delta(1+\delta-\vp_\delta-a\psi_\delta)+\vp_\delta(\xi_0)(s\eta'-\eta'')\\
&> & \vp_\delta(1-\vp_\delta-a\psi_\delta)+\delta \vp_\delta-\vp_\delta(\xi_0)\gamma^2e^{\gamma(\xi-\xi_0)}.
\eess
On the other hand, for such $\xi$,
\bess
\hat\vp(1-\hat\vp-a\psi_\delta)&=& \vp_\delta(1-\hat\vp-a\psi_\delta)+\vp_\delta(\xi_0)
\eta(1-\hat\vp-a\psi_\delta)\\
&\leq& \vp_\delta(1-\vp_\delta-a\psi_\delta)+\vp_\delta(\xi_0)
\eta(1-a\psi_\delta)\\
&=& \vp_\delta(1-\vp_\delta-a\psi_\delta)-\sigma \vp_\delta(\xi_0)
\eta +\epsilon(\xi)\vp_\delta,
\eess
where
\[
\sigma:=a(1-\delta)-1,\;\;
\epsilon(\xi):=\frac{\vp_\delta(\xi_0)}{\vp_\delta(\xi)}\eta(\xi)a(1-\delta-\psi_\delta(\xi))\to 0
\]
uniformly for $\xi\in [\xi_0,\xi_1]$ as $\xi_0\to+\infty$.

Therefore \eqref{tilde-vp} will hold for $\xi\in (\xi_0,\xi_1)$ with large $\xi_0$, provided we can show that
\bes\lbl{3.10}
\frac\delta 2 \vp_\delta(\xi)-\vp_\delta(\xi_0)\gamma^2e^{\gamma(\xi-\xi_0)}\geq -\sigma \vp_\delta(\xi_0)
\eta(\xi)\;\; (\forall \xi\in [\xi_0,\xi_1]).
\ees
By \eqref{vp_delta}, for $\xi\in [\xi_0,\xi_1]$ and large $\xi_0$,
\[
\vp_\delta(\xi)=C_0e^{\tilde\gamma_1\xi}(1+o(1)),\; \vp_\delta(\xi_0)=C_0e^{\tilde\gamma_1\xi_0}(1+o(1)).
\]
Therefore \eqref{3.10} will hold for all large $\xi_0$ if for some $\delta_0\in (0, \delta/2)$ and $\sigma_0\in (0, \sigma)$, we have
\[
\delta_0 e^{\tilde\gamma_1(\xi-\xi_0)}-\gamma^2e^{\gamma(\xi-\xi_0)}+\sigma_0\left[e^{\gamma(\xi-\xi_0)}-1-\gamma(\xi-\xi_0)\right]\geq 0 \; \;\; (\forall \xi\in [\xi_0,\xi_1]).
\]
We show next that this is the case if $\gamma>0$ is chosen sufficiently small.
To this end, we consider
the function
\[
g(z)=g_\gamma(z):=\delta_0 e^{\tilde\gamma_1 z}-\gamma^2e^{\gamma z}+\sigma_0\left(e^{\gamma z}-1-\gamma z\right).
\]
We claim that for small $\gamma>0$, $g_\gamma(z)> 0$ for all $z\geq 0$. 
We calculate
\[
g'(z)=\delta_0\tilde\gamma_1 e^{\tilde\gamma_1 z}-\gamma^3 e^{\gamma z}+\gamma \sigma_0(e^{\gamma z}-1).
\]
Whenever $\gamma^2<\sigma_0$, clearly $g'(z)$ is strictly increasing with $g(0)<0$ and $g(+\infty)=+\infty$. Therefore there exists a unique $z_\gamma\in (0, +\infty)$ such that $g'(z_\gamma)=0$ and $g(z)$ attains its minimum over $[0,+\infty)$ at $z=z_\gamma$.
We show next that
\bes\lbl{3.11}
\lim_{\gamma\to 0} z_\gamma=+\infty,\;\; \lim_{\gamma\to 0}\gamma z_\gamma=0.
\ees
Indeed, for any small $\epsilon>0$, we easily see that
\[
 g'(\epsilon^{-1})\to \delta_0\tilde\gamma_1 e^{\tilde\gamma_1 \epsilon^{-1}}<0,
\;\;
\gamma^{-1}g'(\gamma^{-1} \epsilon)\to \sigma_0(e^\epsilon-1)>0 \mbox{ as } \gamma\to 0.
\]
Therefore for all small $\gamma>0$, 
\[
\epsilon^{-1}<z_\gamma< \epsilon \gamma^{-1},
\]
which clearly implies \eqref{3.11}. To complete our proof, it suffices to show that $g(z_\gamma)>0$ for all small $\gamma>0$.
From $g'(z_\gamma)=0$ we obtain
\[
\delta_0 e^{\tilde\gamma_1 z_\gamma}=\frac{1}{|\tilde\gamma_1|}\left[\sigma_0 \gamma (e^{\gamma z_\gamma}-1)-\gamma^3 e^{\gamma z_\gamma}\right].
\]
Hence, by \eqref{3.11} and the elementary inequality $e^x>1+x \; (\forall x>0)$,
\bess
g(z_\gamma)&=&\frac{1}{|\tilde\gamma_1|}\left[\sigma_0 \gamma (e^{\gamma z_\gamma}-1)-\gamma^3 e^{\gamma z_\gamma}\right]-\gamma^2e^{\gamma z_\gamma}+\sigma_0\left(e^{\gamma z_\gamma}-1-\gamma z_\gamma\right)\\
&>& \frac{1}{|\tilde\gamma_1|}\left(\sigma_0\gamma^2 z_\gamma-\gamma^3e^{\gamma z_\gamma}\right)-\gamma^2 e^{\gamma z_\gamma}\\
&=&\gamma^2\left(\frac{\sigma_0}{|\tilde\gamma_1|}z_\gamma-\frac{\gamma}{|\tilde\gamma_1|}e^{\gamma z_\gamma}-e^{\gamma z_\gamma}\right)\\
&>&0 \mbox{\;\;\;\;\;\;\;  for all small $\gamma>0$.}
\eess
The proof is now complete.
\end{proof}

\begin{proof}[Proof of Lemma \ref{lem3.2}]
For small $\delta>0$, let $(\tilde\vp_\delta(\xi),\psi_\delta(\xi))$ be given by Lemma \ref{lem3.3}.
Let $V(t)$ be the unique solution of
\[
V'=V(1-V),\; \; V(0)=\|v_0\|_\infty.
\]
Then a simple comparison consideration yields $v(t,r)\leq V(t)$ for $t>0$ and $r\geq 0$. Since $\lim_{t\to\infty}
V(t)=1$, we can find $T_0>0$ such that
\[
v(t, r)<1+\frac{\delta}{2} \;\; (\forall t\geq T_0,\;\forall r\geq 0).
\]
Choose $R_0>0$ so that
\[
(1+d)\frac{N-1}{R_0}<\delta.
\]
Then define
\bess
\underline h(t):&=& (s_{\mu,\delta}-\delta)t+R_0+1,\\
\underline u(t,r):&=& \psi_\delta(\underline h(t)-r),\\
\overline v(t,r):&=& \tilde\vp_\delta(\underline h(t)-r).
\eess
Since $\tilde\vp_\delta(-\infty)=1+\delta$, there exists $R_1>\underline h(0)$ such that
\[
\overline v(0,r)=\tilde\vp_\delta(\underline h(0)-r)>1+\frac{\delta}{2}\;\; (\forall r\geq R_1).
\]
Clearly we always have
\[
\overline v(t,r)\geq \tilde\vp_\delta(\xi_1)>0,\; \underline u(t,r)<1-\delta.
\]
By our assumption that spreading of $u$ happens, we can find $T>T_0$ such that
\[
h(t)>R_1,\; u(t,r)>1-\delta,\; v(t,r)<\tilde\vp_\delta(\xi_1)\;\; (\forall t\geq T,\;\forall r\in 
[0, R_1]).
\]
Therefore we have
\[
h(T)>\underline h(0),\; u(T, r)>\underline u(0, r) \; (\forall r\in [0, \underline h(0)]),\; v(T, r)<\overline v(0, r)\; (\forall r\geq 0).
\]
Moreover,
\[
\underline u(t, R_0)<1-\delta<u(t+T, R_0),\; \overline v(t, R_0)\geq \tilde\vp_\delta(\xi_1)>v(t+T, R_0)\; (\forall t>0).
\]
Clearly
\[
\underline h'(t)=s_{\mu,\delta}-\delta<s_{\mu,\delta}=\mu \psi_\delta'(0)=-\mu\underline u_r(t,\underline h(t)) \; (\forall t>0).
\]
Finally, we have, for $t>0$ and $r\in [R_0, \underline h(t))$,
\bess
\underline u_t-d\underline u_{rr}-d\frac{N-1}{r}\underline u_r&=&\left(s_{\mu,\delta}-\delta+d\frac{N-1}{r}\right)\psi_\delta'-d\psi_\delta''\\
&\leq& s_{\mu,\delta}\psi'_\delta-d\psi_\delta''\\
&\leq& r\psi_\delta(1-\psi_\delta-b\tilde\vp_\delta)\\
&=& r\underline u(1-\underline u-b\overline v),
\eess
and for $t>0$ and $r\geq R_0$,
\bess
\overline v_t-\overline v_{rr}-\frac{N-1}{r}\overline v_r&=&\left(s_{\mu,\delta}-\delta+\frac{N-1}{r}\right)\tilde\vp_\delta'-\tilde\vp_\delta''\\
&\geq& s_{\mu,\delta}\tilde\vp'_\delta-\tilde\vp_\delta''\\
&\geq& \tilde\vp_\delta(1-\tilde\vp_\delta-a\psi_\delta)\\
&=& \overline v(1-\overline v-a\underline u).
\eess
Therefore we can apply the comparison principle in \cite{DLin1} (namely Lemma 2.6 there with obvious modifications) to obtain
\bess
\underline h(t)&\leq& h(t+T) \;\; (\forall t>0),\\
\underline u(t,r)&\leq& u(t+T,r)\;\; (\forall t>0,\;\forall r\in [R_0, \underline h(t)]),\\
\overline v(t,r)&\geq& v(t+T,r)\;\; (\forall t>0,\; \forall r\geq R_0).
\eess
It follows that
\[
\liminf_{t\to+\infty}\frac{h(t)}{t}\geq s_{\mu,\delta}-\delta.
\]
Letting $\delta\to 0$, we obtain \eqref{l-bd}.
\end{proof}

\begin{rem}\lbl{remark2}{\rm
Our proof of Theorem \ref{main} above also provides the following estimates for $u$ and $v$:}
\bess
u(t,r)&\geq& \psi_\delta(\underline h(t-T)-r) \; (\forall t>T,\;\forall  r\in [0, \underline h(t-T)]),\\
u(t,r)&\leq & \psi^\delta(\overline h(t-T)-r) \; (\forall t>T,\;\forall r\in [0, h(t)]),\\
v(t,r)&\geq& \vp^\delta(\overline h(t-T)-r)\; (\forall t>T,\;\forall r\in [0,+\infty)),\\
v(t,r)&\leq& \tilde\vp_\delta(\underline h(t-T)-r)\; (\forall t>T,\;\forall r\in [0,+\infty)).
\eess
\end{rem}

\begin{rem}\lbl{v_0-not}
{\rm If  \eqref{v_0} is not satisfied in Theorem \ref{main}, then it is not difficult to find examples of $v_0>0$ such that spreading of $u$ happens and
$\lim_{t\to+\infty} h(t)/t=k_\mu>s_\mu$, where $k_\mu$ is the spreading speed of a single species free boundary problem, obtained by letting $v\equiv 0$ in \eqref{p1}.
We leave the detailed construction of such examples to the interested reader.
}
\end{rem}

\end{document}